\documentclass{siamltex}

\usepackage[text={5.125in,8.25in},centering]{geometry}
\usepackage{microtype}
\usepackage{xspace}      % use \xspace after macros
\usepackage{amsmath}
\usepackage{amssymb}
\usepackage{enumitem}
\usepackage{overpic}
\usepackage{hyperref}
\usepackage{graphicx}
\usepackage{cite}        % sorts numbered citations
\usepackage{braket}
\usepackage{subdepth}
\usepackage{color}

\usepackage{pgfplots}\pgfplotsset{compat=1.4}

\usepackage{url}
\usepackage{graphicx}          %
\RequirePackage[OT1]{fontenc}
\RequirePackage{amsmath,amssymb}

\definecolor{softgray}{rgb}{0.92,0.92,0.95}
\definecolor{softblue}{rgb}{0.90,0.92,1.00}
\definecolor{lightgray}{rgb}{0.12,0.12,0.55}

\definecolor{theframe}{gray}{0.75}
\definecolor{theblue} {rgb}{0.02,0.04,0.48}
\definecolor{thegrey} {gray}{0.5}
\definecolor{theshade}{gray}{0.98}
\definecolor{thered}  {rgb}{0.00,0.00,0.00}
\definecolor{thegreen}{rgb}{0.3,0.3,0.3}

\usepackage{color}
\definecolor{softblue}{rgb}{0.90,0.92,1.00}
\newtheorem{lem}{Lemma}[section]
\newtheorem{remark}[lem]{Remark}
\newtheorem{algor}[lem]{Algorithm}
\newcommand{\B}[1]{{\bf #1}}
\newcommand{\Sc}[1]{{\mathcal{#1}}}
\newcommand{\R}[1]{{\rm #1}}
\newcommand{\mB}[1]{{\mathbb{#1}}}
\let\Newset\Set
\renewcommand{\set}[2]{\Newset{#1|#2}}

\title{Robust and Trend-following
Student's t-Kalman Smoothers}
 \author{%
    Aleksandr Y. Aravkin\thanks{IBM T.J. Watson Research Center, Yorktown Heights, NY, 10598
      (saravkin@us.ibm.com)}.
    \and James V. Burke\thanks{Mathematics Dept., University
      of Washington, Seattle, WA 98195 (burke@math.washington.edu).}
    \and Gianluigi Pillonetto
    \thanks{Control and Dynamic Systems Department of Information Engineering at the University of Padova, Padova, Italy (giapi@dei.unipd.it)}
    \hfill\today}
\begin{document}

\maketitle

%\title{Robust and Trend-following\\
%Student's t-Kalman Smoothers}

%\author{Aleksandr~Y.~Aravkin,
%        James~V.~Burke,
%        and~Gianluigi~Pillonetto,% <-this % stops a space
%\thanks{A.Y. Aravkin is a Research Staff Member at IBM Watson Research Center, 
%Yorktown Heights, NY.}% <-this % stops a space
%\thanks{J.V. Burke is with the Mathematics Department at the University of Washington.}% <-this % stops a space
%\thanks{G. Pillonetto is with the Control and Dynamic Systems Department of Information Engineering at the University of Padova, Italy.}}
%
%% The paper headers
%\markboth{Submitted to IEEE Transactions on Signal Processing}%
%{Shell \MakeLowercase{\textit{et al.}}: Student's t Kalman Smoothers}

\begin{abstract}
We present a Kalman smoothing framework based on 
modeling errors using the heavy tailed Student's t distribution,
along with algorithms, convergence theory, open-source general implementation, 
and several important applications. 
The computational effort per iteration grows linearly with the length of the time series, 
and all smoothers allow nonlinear process and measurement models. 

Robust smoothers form an important subclass of smoothers within this framework.  
These smoothers work in situations where measurements are highly contaminated
by noise or include data unexplained by the forward model.  
Highly robust smoothers are developed by modeling {\it measurement} 
errors using the Student's t distribution, 
and outperform the recently proposed
$\ell_1$-Laplace smoother in extreme situations
with data containing 20\% or more outliers.

A second special application we consider in detail
allows tracking sudden changes in the state. 
It is developed by modeling {\it process} noise using 
the Student's t distribution, and the resulting smoother 
can track sudden changes in the state.

These features can be used separately or in tandem, 
and we present a general smoother algorithm and open source implementation, 
together with convergence analysis that covers a wide range of smoothers. 
A key ingredient of our approach is a technique to deal with the 
non-convexity of the Student's t loss function.  
Numerical results for linear and nonlinear models illustrate the performance
of the new smoothers for robust and tracking applications, 
as well as for mixed problems that have both types of features.

\end{abstract}

\section{Introduction}

The Kalman filter is an efficient recursive algorithm for estimating
the state of a dynamic system \cite{kalman}. Traditional
formulations are based on $\ell_2$ penalties on model deviations,
and are optimal under assumptions of linear dynamics and
Gaussian noise. Kalman filters are used in a wide array of applications
including navigation, medical technologies, and econometrics
\cite{Chui2009,West1991,Spall03}. Many of these problems are nonlinear, and
may require smoothing over past data in both online and offline
applications to significantly improve estimation performance \cite{Gelb}.\\
This paper focuses on two important areas in Kalman smoothing:
robustness with respect to outliers in measurement data, and
improved tracking of quickly changing system dynamics. Robust
filters and smoothers have been a topic of significant interest since the
1970's, e.g. see \cite{Schick1994}. 
%A major body of work is devoted to designing 
%filters robust to uncertainty in the {\it dynamical model}; 
%see~\cite{Petersen1999}
%and the numerous references within. 
Recent efforts have focused
on building smoothers that are robust to outliers in the data~\cite{AravkinL12011,
AravkinIFAC,Farahmand2011}, using convex loss functions such as
$\ell_1$, Huber or Vapnik, in place of the $\ell_2$ penalty~\cite{Hastie01}.

There
have also been recent efforts to design smoothers able to better
track fast system dynamics, e.g. jumps in the state values. A
contribution can be found in \cite{Ohlsson2011} where the Laplace distribution,
rather than the Gaussian, is used to model transition
noise. This introduces an $\ell_1$ penalty on the state evolution
in time, resulting in an estimator interpretable as a dynamic
version of the well known LASSO procedure
\cite{Lasso1996}.

For known dynamics, all of the smoothers mentioned above can be derived
by modeling the
process and the measurement noise using log-concave densities, taking the form
\begin{equation}
\label{logconcave}
\B{p}(\cdot) \propto \exp(-\rho(\cdot)), \quad \rho \text{ convex}\;.
\end{equation}
Formulations exploiting~\eqref{logconcave} are nearly ubiquitous,
in part because they correspond to convex optimization problems
in the linear case. However, in order to model a regime with large outliers
or sudden jumps in the state, we want to look beyond~\eqref{logconcave}
and allow heavy-tailed densities, i.e. distributions
whose tails are not exponentially bounded.
All such distributions necessarily have non-convex loss
functions~\cite[Theorem 2.1]{AravkinFHV:2012}. 

Several interesting candidates are possible, but in this contribution we focus on
the Student's t-distribution for its convenient properties in the context
of the applications we consider.
The Student's t-distribution was successfully applied to a
variety of robust inference applications in \cite{Lange1989},
and is closely related to re-descending influence functions \cite{Hampel}.

In this work, we propose a smoothing framework for 
several applications, including robust and trend smoothing. 
The T-Robust smoother is derived from a dynamic system
with {\it output noise} modeled by the Student's t-distribution. This is a
further robustification of the estimator proposed in \cite{AravkinL12011}, 
which uses the Laplace density. 
The re-descending influence function
of the Student's t guarantees that outliers in the measurements have
less of an effect on the smoothed estimate than any convex loss function.
In practice, the T-Robust smoother performs better than
\cite{AravkinL12011} for cases with a high proportion of outliers.
The T-Trend smoother is similarly derived starting from a
dynamic system with {\it transition noise} modeled by
the Student's t-distribution. This allows
T-Trend to better track sudden changes in the state. 
One may consider using both aspects simultaneously; 
in addition, practitioners need the ability 
to distinguish between different measurements based 
on prior information of measurement fidelity, and 
between different states based on prior knowledge of trend stability.

In the context of Kalman filtering/smoothing, the idea of using
Student's t-distributions to model the system noise for robust
and tracking applications was first proposed in \cite{Fahr1998}.
However, our work differs from that approach in some important aspects.
First, our analysis includes nonlinear measurement and process models.
Second, we provide a novel approach to overcome the non-convexity
of the Student's t-loss function. Third, the approach we propose 
can be used to solve {\it any} smoothing problem
that uses Student's t modeling for any process or measurement components. 

The basic approach differs significantly from the one proposed in \cite{Fahr1998}.
\cite{Fahr1998} proposes using the random information matrix (i.e.
full Hessian) when possible, or its expectation (Fisher information) when
the Hessian is indefinite.
Instead, we propose a modified Gauss-Newton method which builds
information about the curvature of the Student's t-log likelihood
into the Hessian approximation, and is guaranteed to be positive definite.
As we show in Section~\ref{Algorithm},
the new approach is provably convergent,
and unlike the approach in \cite{Fahr1998} uses information
about the relative sizes of the residuals
in computing descent directions. These differences make it
more stable than methods using random information, and more
efficient than methods using Fisher information.\\
The major computational tradeoff in using non-convex penalties
is that the loss function in the convex case is used directly~\cite{AravkinL12011}, i.e. is not approximated,
whereas in the nonconvex case, the loss function must be iteratively
approximated with a local convex approximation. This requires a fundamental
extension of the convergence analysis.  

A conference proceeding previewing this paper appears in~\cite{SYSID2012tks}.
In the current work, we present a general smoothing framework that 
includes the two smoothers presented in~\cite{SYSID2012tks} as special cases, 
together with a generalized convergence theory that covers the entire 
range of smoothers under discussion. 
We also provide an open-source implementation 
of the general algorithm~\cite{ckbs}, 
with a simple interface that enables the user 
to customize which residual or innovation components to model 
using the Student's t penalty. 
Using this implementation, we present an expanded experimental section, 
and new experiments that show how robust and trend smoothing can be done 
simultaneously. Finally, we apply the smoothers to real data. 

%The fast smoothing procedures proposed here also
%allow the efficient estimation of hyperparameters such as degrees of
%freedom (e.g. cross-validation techniques using the EM algorithm are
%discussed in \cite{Fahr1998}).  \\
The paper is organized as follows.
In Section \ref{StudentT-KF}, we introduce
the multivariate Student's t-distribution, review its advantages
for error modeling over log-concave distributions, and introduce 
the dynamic model class of interest for Kalman smoothing.  
In Section \ref{GenT}, we describe a statistical modeling framework, where we 
can use Student's t to model any process or measurement residual components. 
We describe all objectives that can arise this way, and provide a comprehensive 
method for obtaining approximate second order information for these objectives.
In Section \ref{SpecialCases}, we provide details for three important special 
smoothers:  T-Robust (robust against large measurement noise), T-Trend (able to 
follow sharp changes in the state), and the Double-T smoother (incorporates both aspects).  
In Section \ref{Algorithm}, we present the algorithm and a convergence
theory for the entire framework, which also  
extends the convergence theory developed in~\cite{AravkinL12011}. 
In Section~\ref{SimulationRobust}, we present numerical 
experiments that illustrate the behavior of all three special smoothers,   
include illustrations of linear and nonlinear models, 
and results for real and simulated data.
We end the paper with concluding remarks.
\section{Error Modeling with Student's t}
\label{StudentT-KF}

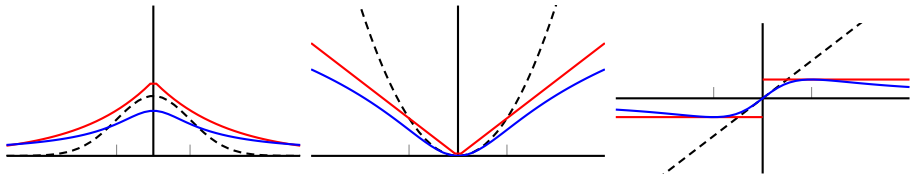
\begin{figure} \label{GLT-KF}
\centering
%%%%%%%%%%%%%%%%%%%%%
% Densities
%%%%%%%%%%%%%%%%%%%%%
\begin{tikzpicture}
  \begin{axis}[
    thick,
    width=.3\textwidth, height=2cm,
    xmin=-4,xmax=4,ymin=0,ymax=1,
    no markers,
    samples=50,
    axis lines*=left, 
    axis lines*=middle, 
    scale only axis,
    xtick={-1,1},
    xticklabels={},
    ytick={0},
    ] 
%\addplot[red,domain=-2:-1,densely dashed]{-x-.5};
\addplot[domain=-4:+4,densely dashed]{exp(-.5*x^2)/sqrt(2*pi)};
 \addplot[red, domain=-4:+4]{0.5*exp(-.5*abs(x))};
  \addplot[blue, domain=-4:+4]{0.3*exp(-.5*ln(1 + x^2))};
%\addplot[red,domain=+1:+2,densely dashed]{x-.5};
%\addplot[blue,mark=*,only marks] coordinates {(-1,.5) (1,.5)};
  \end{axis}
\end{tikzpicture}
%%%%%%%%%%%%%%%%%%%%%
% Penalties
%%%%%%%%%%%%%%%%%%%%%
\begin{tikzpicture}
  \begin{axis}[
    thick,
    width=.3\textwidth, height=2cm,
    xmin=-3,xmax=3,ymin=0,ymax=2,
    no markers,
    samples=50,
    axis lines*=left, 
    axis lines*=middle, 
    scale only axis,
    xtick={-1,1},
    xticklabels={},
    ytick={0},
    ] 
%\addplot[red,domain=-2:-1,densely dashed]{-x-.5};
\addplot[domain=-3:+3,densely dashed]{.5*x^2};
 \addplot[red, domain=-3:+3]{.5*abs(x)};
  \addplot[blue, domain=-3:+3]{.5*ln(1 + x^2)};
%\addplot[red,domain=+1:+2,densely dashed]{x-.5};
%\addplot[blue,mark=*,only marks] coordinates {(-1,.5) (1,.5)};
  \end{axis}
\end{tikzpicture}
%%%%%%%%%%%%%%%%%%%%%
% Influence Functions
%%%%%%%%%%%%%%%%%%%%%
\begin{tikzpicture}
  \begin{axis}[
    thick,
    width=.3\textwidth, height=2cm,
    xmin=-3,xmax=3,ymin=-2,ymax=2,
    no markers,
    samples=50,
    axis lines*=left, 
    axis lines*=middle, 
    scale only axis,
    xtick={-1,1},
    xticklabels={},
    ytick={0},
    ] 
%\addplot[red,domain=-2:-1,densely dashed]{-x-.5};
\addplot[domain=-3:3,densely dashed]{x};
 \addplot[red, domain=-3:0]{-.5};
  \addplot[red, domain=0:3]{.5};
\addplot[blue, domain=-3:3]{x/(1 + x^2)};
%\addplot[red,domain=+1:+2,densely dashed]{x-.5};
%\addplot[blue,mark=*,only marks] coordinates {(-1,.5) (1,.5)};
  \end{axis}
\end{tikzpicture}
    \caption{Gaussian, Laplace, and Student's t Densities, Corresponding Negative Log Likelihoods, and Influence Functions.}
%\caption{
%Densities $\mathbf{p}(\textcolor{blue}{v})$, penalties $-\ln\mathbf{p}(\textcolor{blue}{v})$ ,
 %and influence fns $-\frac{d}{d \textcolor{blue}{v} }\ln\mathbf{p}(\textcolor{blue}{v})$. }
\end{figure}

%\begin{figure*} \label{GLT-KF}
%  \begin{center}
%  \begin{tabular}{ccc}
%  \includegraphics[scale=0.39]{TFigures/L1L2STdens-1.pdf}
%\hspace{.1in}
%\includegraphics[scale=0.39]{TFigures/L1L2STobj-1.pdf}
%\hspace{.1in}
%\includegraphics[scale=0.39]{TFigures/L1L2STinf-1.pdf}
%    \end{tabular}
%    \caption{Gaussian, Laplace, and Student's t Densities, Corresponding Negative Log Likelihoods, and Influence Functions (for scalar $v_k$).}
%     \end{center}
%\end{figure*}

For a vector \(u \in \mB{R}^n\)
and any positive definite matrix \(M \in \mB{R}^{n\times n}\), let
\(\|u\|_M := \sqrt{u^\R{T}Mu}\).
We use the following generalization of the Student's
t-distribution:
\begin{eqnarray}
\label{StudentDensity}
\B{p}(v_k|\mu)
&=&
 \frac{\Gamma (\frac{s + m}{2})}
{\Gamma(\frac{s}{2})\det[\pi s R]^{1/2}}
\left(1 + \frac{\|v_k - \mu\|_{R^{-1}}^2}{s}\right)^{\frac{-(s + m)}{2}}
\end{eqnarray}
where $\mu$ is the mean, $s$ is the degrees of freedom,
$m$ is the dimension of the vector $v_k$,
and $R$ is a positive definite matrix.
A comparison of this distribution
with the Gaussian and Laplacian distribution
appears in Figure \ref{GLT-KF}. Note that the Student's t-distribution
has much heavier tails than the others, and that its influence
function is re-descending, see \cite{Mar} for a discussion of
influence functions. This means that as we pull a
measurement further and further away, its `influence'
decreases to 0, so it is eventually ignored by the model.
Note also that the $\ell_1$-Laplace is peaked at 0, while the
Student's t-distribution is not, and so a Student's t-fit
will not in general drive residuals to be exactly $0$.

Before we proceed with the Kalman smoothing application, 
we review a result from~\cite{AravkinFHV:2012},
illustrating the fundamental modeling advantages of heavy tailed distributions:

\begin{theorem}
\label{UniversalLowerBound}
Consider any scalar density $p$ arising from a symmetric convex
coercive and differentiable penalty $\rho$ via $p(x) = \exp(-\rho(x))$, 
and take any point $t_0$ with $\rho'(t_0) = \alpha_0 > 0$. Then for
all $t_2 > t_1\geq t_0$, the conditional tail distribution induced by $p(x)$
satisfies
\begin{equation}
\label{memoryFreeIneq}
\Pr(|y| > t_2 \mid |y| > t_1) \leq \exp(-\alpha_0[t_2 - t_1])\;.
\end{equation}
%\hfill\QED
\end{theorem}
When $t_1$ is large, the condition $|y| > t_1$ indicates that 
we are looking at an outlier. However, as shown by the theorem, {\it any}
log-concave statistical model treats the outlier conservatively, 
dismissing the chance that $|y|$ could be significantly bigger than $t_1$. 
Contrast this behavior with that of the Student's t-distribution.
When $\nu = 1$, the Student's t-distribution is simply the Cauchy
distribution, with a density proportional to $1/(1 + y^2)$.  Then we
have that
\[
 \lim_{t\to\infty} \Pr(|y|>2t \mid |y|>t) =
 \lim_{t\to\infty} \frac{\frac{\pi}{2}-\arctan(2t)}{\frac{\pi}{2} - \arctan(t)}
 = \frac{1}{2}.
\] 
Heavy tailed distributions thus provide a fundamental advantage
in cases where outliers may be particularly large, or, in the second
application we discuss, very sudden trend changes may be present. 

We now turn to the Kalman smoothing framework.  
We use the following general model for the underlying dynamics:
for $k = 1 , \ldots , N$
\begin{equation}
\label{NonlinearStudentModel}
\begin{array}{ccc}
    x_{k} & = & g_k(x_{k-1}) + w_k
    \\
    z_k   & = & h_k(z_k) + v_k
\end{array}
\end{equation}
with initial condition $g_1(x_0) = g_0 + w_1$, with $g_0$ a known constant, and
where $g_k: \mB{R}^n \rightarrow \mB{R}^n$ are known smooth process functions,
and $h_k: \mB{R}^n \rightarrow \mB{R}^{m}$ are known smooth
measurement functions. Moreover, $w_k$ and $v_k$ are mutually independent, 
and with known covariance matrices $Q_k \in \mB{R}^{n\times n}$ and 
$R_k \in \mB{R}^{m \times m}$, respectively. 
Note that here we assume all the measurement vectors
have consistent dimension $m$. There is no loss of generality compared to the 
standard model where the dimensions depend on $k$, since 
any measurement vector can be augmented to a standard size $m$, and then 
the phantom measurements can be disabled using the modeling interface
(by setting corresponding columns and rows of $R_{k}^{-1}$ to $0$.) 

We now briefly explain how to use Student's t error modeling to 
design smoothers with two important characteristics. 
In order to obtain smoothers that are robust to heavily contaminated data, 
the vector $v_k \in \mB{R}^{m(k)}$
can be modeled zero-mean Student's t measurement noise \eqref{StudentDensity}
of known covariance $R_k \in \mB{R}^{m(k) \times m(k)}$ and degrees of freedom $s$.
To design smoothers that can track sudden changes in the state, 
the process residuals $w_k$ are modeled using Student's t noise. 
These features may be employed separately or in tandem, 
and we always assume that the vectors $ \{ w_k \} \cup \{ v_k \} $
are all mutually independent.

In the next section, we design a smoother that finds the MAP estimates of $\{x_k\}$
for a general  formulation, where Student's t or least squares modeling 
can be used for any or all process and measurement residuals. 
We then specialize it to recover the applications discussed above. 

%%%%%%%%%%%%%%%%%%%%%%%%%%%%%%%%%%%%%%%%%%
\section{Generalized Smoothing Framework}
\label{GenT}
%%%%%%%%%%%%%%%%%%%%%%%%%%%%%%%%%%%%%%%%%%
Given a sequence of column vectors $\{ u_k \}$
and matrices $ \{ T_k \}$ we use the notation
\[
\R{vec} ( \{ u_k \} )
=
\begin{bmatrix}
u_1 \\ u_2  \\ \vdots \\ u_N
\end{bmatrix}
\; , \;
\R{diag} ( \{ T_k \} )
=
\begin{bmatrix}
T_1    & 0      & \cdots & 0 \\
0      & T_2    & \ddots & \vdots \\
\vdots & \ddots & \ddots & 0 \\
0      & \cdots & 0      & T_N
\end{bmatrix} .
\]
We also make the following definitions:
\[
\begin{array}{rcl}
R       & = & \R{diag} ( \{ R_k \} )
\\
Q       & = & \R{diag} ( \{ Q_k \} )
\\
x       & = & \R{vec} ( \{ x_k \} )
\end{array}
\; , \;
\begin{array}{rcl}
w( x )    & = & \R{vec} ( \{ x_k - g_k ( x_{k-1} ) \} )
\\
v ( x )  & = & \R{vec} ( \{ z_k - h_k( x_k ) \} ) .
\end{array}
\]
In the most general case, we suppose that 
any of the components $w_k^i$
or $v_k^i $ components can be modeled 
either using Gaussian or Student's t distributions. 

For the sake of modeling clarity, 
assume that {\it subcomponents} of measurement and innovation
residuals are consistently modeled across time points $k$; 
this gives the user the ability to select which {\it subvectors} 
of process and measurement residuals to model using Student's t, 
but not to assign different penalties to different time points. 

Denote by $w_k^G$ and $w_k^S$ the subvectors of the innovation 
residuals $w_k$, and denote by $v_k^G$ and $v_k^S$ the subvectors of the measurement residuals $v_k$ 
that are to be modeled using the Gaussian and Student's t distributions, 
respectively. Assume that all of these subvectors are mutually 
independent, and denote the corresponding covariance submatrices 
by $Q_k^G$, $Q_k^S$, $R_k^G$, and $R_k^S$. 
Maximizing the likelihood for this model 
is equivalent to minimizing the associated negative log likelihood
\[
-\ln \B{p}(\{\nu_k^G\},\{\nu_k^S\}, \{w_k^G\},\{w_k^S\}),
\]
which can be explicitly written as follows: 
\begin{equation}
\label{fullObjectiveGeneral}
\begin{aligned}
&\sum^N_{k=1}s\ln
\left[1 + \frac{\|v_k^S\|_{(R_{k}^S)^{-1}}^2}{s}\right]
+
\|v_k^G\|_{(R_k^{G})^{-1}}^2 
+
r\ln
\left[1 + \frac{\|w_k^S\|_{(Q_{k}^S)^{-1}}^2}{r}\right]
+
\|w_k^G\|_{(Q_k^G)^{-1}}^2 
\end{aligned}
\end{equation}
where $s$ and $r$ are degree of freedom parameters
corresponding to $v_k^S$ and 
$w_k^S$.

A first-order accurate
affine approximation to our model
with respect to direction
$d = \R{vec}\{d_k\}$
near a fixed state sequence $x$ is given by
\[
\begin{array}{lll}
\tilde w( x;d )
& = &
\R{vec} ( \{ x_k - g_k(x_{k-1}) - g_k^{(1)} (x_{k-1}) d_k \} ),
\\
\tilde v ( x;d  )
& = &
\R{vec} ( \{ z_k -  h_k(x_k) - h_k^{(1)} (x_k) d_k \} ).
\end{array}
\]
Set $Q_{N+1} = I_n$ and $g_{N+1} ( x_N ) = 0$
(where $I_n$ is the $n \times n$ identity matrix)
so that the formulas are also valid for $k = N+1$.

We minimize the nonlinear nonconvex objective in (\ref{fullObjectiveGeneral})
by iteratively solving quadratic programming (QP) subproblems of the form:
\begin{equation}
\label{QuadraticSubproblemOne}
\begin{array}{lll}
\mbox{min}
    &\frac{1}{2}d^\R{T}Cd + a^\R{T}d
    \quad \mbox{w.r.t} \;  d \in \mB{R}^{nN} ,
\end{array}
\end{equation}
where $a$ is the gradient of objective~\eqref{fullObjectiveRobust} with respect
to $x$ and $C$ has the form
\begin{equation}
\label{hessianApprox}
C
=
\begin{bmatrix}
C_1 + H_1 & A_2^\R{T} & 0 & \\
A_2 & C_2 + H_2 & A_3^\R{T} & 0 \\
0 & \ddots & \ddots& \ddots & \\
& 0 & A_N & C_N + H_N
\end{bmatrix} ,
\end{equation}
Note that this matrix is symmetric block tridiagonal. This structure is essential
to the computational results for a wide variety of Kalman filtering and smoothing 
algorithms; it was noted early on in~\cite{Wright1990,Fahr1991}.

In order to fully describe $C_k$ and $A_k$, first let 
$\Sc{W}^G$, $\Sc{W}^S$ denote the indices
associated to all subvectors $w_k^G$ and $w_k^S$
within $w_k$. 
For example, if the Student's t density is used for all 
measurement residuals, and the Gaussian penalty
is used for all process residuals, then 
$\Sc{W}^G = \{1, \dots, n\}$, $\Sc{W}^S = \emptyset$. 

Now define with $A_k, C_k, H_k \in \mB{R}^{n\times n}$ as follows:
\begin{eqnarray}
\nonumber
&A_k(\Sc{W}^S, \Sc{W}^S) 
\label{Ak}
&=
 -\frac{r (Q_k^S)^{-1}(g^{(1)}_k)^S }{r + \|w_k^S\|_{(Q_k^S)^{-1}}^2}\\
&A_k(\Sc{W}^G, \Sc{W}^G) 
 &= 
-(Q_k^G)^{-1}(g^{(1)}_{k})^G\\
\nonumber
&C_k(\Sc{W}^G, \Sc{W}^G) &= ((g^{(1)}_{k+1})^G)^\R{T}(Q^G_{k+1})^{-1}(g^{(1)}_{k+1})^G +  (Q_k^G)^{-1}\\
&C_k(\Sc{W}^S, \Sc{W}^S) 
\label{Ck}
&= 
\frac{r((g^{(1)}_{k+1})^S)^\R{T}(Q^S_{k+1})^{-1}(g^{(1)}_{k+1})^S}
{{r + \|w_{k+1}^S\|_{(Q_{k+1}^S)^{-1}}^2}}
+  \frac{r (Q_k^S)^{-1}}{r + \|w_k^S\|_{(Q_k^S)^{-1}}^2} \\
\nonumber
&H_k
&=
\label{Hk}
\frac{s((h_k^{(1)})^S)^\R{T}(R_k^S)^{-1}(h_k^{(1)})^S}
{(s + \|v_k^S\|_{(R_k^S)^{-1}}^2)} 
+ ((h_k^{(1)})^G)^\R{T}(R_k^G)^{-1}(h_k^{(1)})^G.
\end{eqnarray}
The entries of $A_k$ and $C_k$ not explicitly defined in~\eqref{Ak} and~\eqref{Ck}
are set to $0$. 

The Hessian approximation terms $H_k$ in~\eqref{Hk} are motivated in Section \ref{Algorithm},
and are crucial to both practical performance and theoretical convergence analysis.
The solutions to the subproblem \eqref{QuadraticSubproblemOne} have the form
$d = -C^{-1}a$,
and can be found in an efficient
and numerically stable manner in $O(n^3N)$ steps,
since $C$ is tridiagonal and positive definite (see \cite{Bell2008}).

\section{Special cases}
\label{SpecialCases}
We know show how the general framework of the previous section can be 
specialized to obtain three smoothers. 
The first two are T-Robust and T-Trend, which are presented in~\cite{SYSID2012tks}.
The third is a new smoother where {\it all} residuals and innovations are modeled
using Student's t.  

The objective corresponding to T-Robust is obtained from~\eqref{fullObjectiveGeneral} by taking $w_k^G = w_k$, 
$w_k^S = 0$, $v_k^G = 0$, $v_k^S = v_k$: 
%%%%%%%%%%%%%%%%%%%%%%%%%%%%%%%%%%%%%%
% T-Robust
\begin{eqnarray}
\label{fullObjectiveRobust}
\frac{1}{2}\sum^N_{k=1}s\ln
\left[1 + \frac{\|v_k\|_{R_{k}^{-1}}^2}{s}\right]
+
\|w_k\|_{Q_k^{-1}}^2.
\end{eqnarray}
The terms $A_k, C_k, H_k$ in~\eqref{Ak}---\eqref{Hk} become
\begin{eqnarray}
\nonumber
A_k
&=&
-Q_k^{-1}g^{(1)}_{k}\; , \;\\
\nonumber
C_k
&=&
Q_k^{-1} + (g^{(1)}_{k+1})^\R{T}Q^{-1}_{k+1}g^{(1)}_{k+1}\; ,\\
\label{TrobH}
H_k
&=&
\frac{s(h_k^{(1)})^\R{T}R_k^{-1}h_k^{(1)}}
{(s + \|v_k\|_{R^{-1}_k}^2)}\; .
\end{eqnarray}
% T Robust
%%%%%%%%%%%%%%%%%%%%%%%%%%%%%%%%%%%%%%
%
%
%
%%%%%%%%%%%%%%%%%%%%%%%%%%%%%%%%%%%%%%
% T-Trend

The objective corresponding to T-Trend is obtained from~\eqref{fullObjectiveGeneral}
by taking  $w_k^G = 0$, 
$w_k^S = w_k$, $v_k^G = v_k$, $v_k^S = 0$:
\begin{equation}
\label{fullObjectiveTrend}
\frac{1}{2}\sum^N_{k=1}r
\ln \left[1 + \frac{\|w_k\|_{Q_{k}^{-1}}^2}{r}\right]
+
\|v_k\|_{R_k^{-1}}^2.
\end{equation}
The terms $A_k, C_k, H_k$ in~\eqref{Ak}---\eqref{Hk} become
\begin{eqnarray}
\nonumber
A_k
&=&
-\frac{rQ_k^{-1}g^{(1)}_{k}}{r + \|w_{k}\|_{Q^{-1}_k}^2},\\
\nonumber
C_k
&=&
\frac{rQ_k^{-1}}{r + \|w_k\|_{Q^{-1}_k}^2} +
\frac{r(g^{(1)}_{k+1})^\R{T}Q^{-1}_{k+1}g^{(1)}_{k+1}}
{r + \|w_{k+1}\|_{Q^{-1}_{k+1}}^2},\\
\label{TrendH}
H_k
&=&
 (h_k^{(1)})^\R{T}R_k^{-1}h_k^{(1)}.
\end{eqnarray}
% T Trend
%%%%%%%%%%%%%%%%%%%%%%%%%%%%%%%%%%%%%%
%
%
%
%%%%%%%%%%%%%%%%%%%%%%%%%%%%%%%%%%%%%%
% All Student

Finally, we can apply Student's t to all process and measurement
residuals by taking $w_k^G = 0$, 
$w_k^S = w_k$, $v_k^G = 0$, $v_k^S = v_k$ to obtain 
\begin{equation}
\label{fullObjectiveStudent}
\frac{1}{2}
\sum^N_{k=1}r_k
\ln \left[1 + \frac{\|w_k\|_{Q_{k}^{-1}}^2}{r_k}\right]
+
s_k\ln
\left[1 + \frac{\|v_k\|_{R_{k}^{-1}}^2}{s_k}\right]
\end{equation}
The terms $A_k, C_k, H_k$ in~\eqref{Ak}---\eqref{Hk} become
\begin{eqnarray}
\nonumber
A_k
&=&
-\frac{rQ_k^{-1}g^{(1)}_{k}}{r + \|w_{k}\|_{Q^{-1}_k}^2},\\
\nonumber
C_k
&=&
\frac{rQ_k^{-1}}{r + \|w_k\|_{Q^{-1}_k}^2} +
\frac{r(g^{(1)}_{k+1})^\R{T}Q^{-1}_{k+1}g^{(1)}_{k+1}}
{r + \|w_{k+1}\|_{Q^{-1}_{k+1}}^2},\\
\label{FullH}
H_k
&=&
\frac{s(h_k^{(1)})^\R{T}R_k^{-1}h_k^{(1)}}
{(s + \|v_k\|_{R^{-1}_k}^2)}\; .
\end{eqnarray}
% All Student
%%%%%%%%%%%%%%%%%%%%%%%%%%%%%%%%%%%%%%

%%%%%%%%%%%%%%%%%%%%%%%%%%%%%%%%%%%%%%%%%
\section{Algorithm and Global Convergence}
\label{Algorithm}
%%%%%%%%%%%%%%%%%%%%%%%%%%%%%%%%%%%%%%%%%

When models $g_k$ and $h_k$ are linear,
we can compare the algorithmic scheme proposed
in the previous sections with the method
in \cite{Fahr1998}.
The latter uses the random information matrix (random Hessian)
in place of the matrix $C$ defined above, and recommends
using the expected (Fisher) information when the full Hessian
is indefinite.
When the densities for $w_k$ and $v_k$
are Gaussian, this is equivalent to using Newton's
method when possible, and using Gauss-Newton when the Hessian is indefinite.
In general, using the expected information is known as the method
of Fisher's scoring.
In the Student's t-case, the scalar Fisher information matrix is computed
in~\cite{Lange1989} to be
\begin{equation}
\label{ExpectedFisher}
\frac{s+1}{s+3}\sigma^{-2}\;,
\end{equation}
where $\sigma^2$ is the variance and $s$ is the degrees of freedom.
The authors of~\cite{Fahr1998} proposed using~\eqref{ExpectedFisher} as the Hessian approximation
when the full Hessian is indefinite.
Implementing this approach would effectively replace
the terms $\|w_k\|_2^2$ or $\|v_k\|_2^2$, present
in the denominators of $H_k$ and $A_k$ (see \ref{TrobH} and \ref{TrendH}),
with terms that depend only on $s_k$ and $r_k$, the degrees of freedom.
So while the random information (Hessian)
matrix can become indefinite,
the Fisher information is insensitive
to outliers, and fails to down-weigh their contributions to the Hessian approximation. 

%To overcome these drawbacks, it is necessary to
%provide the algorithm with the information about the magnitude of the
%residuals $\|w_k\|$ and $\|v_k\|$, as a function
%of the iteration number,
%so that it can curtail their contribution to the
%model updates.
To overcome these drawbacks, and find a middle ground between using
the full Hessian and using a very rough approximation, we propose
a Gauss-Newton method
that is able to incorporate the relative size
information of the residuals into the Hessian approximation.
In the rest of this section
we provide the details for the application of this method
and a proof of convergence.

As in~\cite{AravkinL12011}, the convergence theory
is based upon the versatile convex-composite techniques developed in~\cite{Burke85}.
We begin by choosing the convex-composite structure for objective~\eqref{fullObjectiveGeneral}.
We write it in the convex-composite form \(K= \rho \circ F\),
with smooth $F$ and convex $\rho$:
\begin{eqnarray}
\rho \left( \begin{array}{c} c \\ u \end{array} \right)
\label{rhoDef}
&=&
c
+
\frac{1}{2}\|u\|_{B^{-1}}^2
+
\delta_{\mathbb{R}_+}(c)
\\
F(x)
\label{FDef}
&=&
\left( \begin{array}{c}  f(x) \\ \begin{bmatrix} w^G(x) \\ v^G(x) \end{bmatrix} \end{array} \right)
\\
f(x)
\label{fDef}
&=&
 \frac{1}{2}\sum^N_{k=1}s\ln
\left[1 + \frac{\|v_k^S\|_{(R_{k}^S)^{-1}}^2}{s}\right]
+
\sum^N_{k=1}r\ln
\left[1 + \frac{\|w_k^S\|_{(Q_{k}^S)^{-1}}^2}{r}\right]\;.
\end{eqnarray}

Note that the range of $f$ is $\B{R_+}$, and $\rho$
is coercive on its domain. 
%$\mB{R}_+ \times \mB{R}^{nN}$ and
%$\mB{R}_+ \times \mB{R}^{M}$, respectively, where $M = \sum m(k)$.
The terms indexed with superscript $S$ in~\eqref{Ck} and~\eqref{Hk} 
combine to form a positive definite approximation to the Hessian of $f$. 
To see this, consider the scalar function
\[
\kappa(x):=\frac{1}{2}\ln(1 + x^2/r)\;.
\]
The second derivative of this function in $x$ is given by
\begin{equation}
\label{LogHessian}
\frac{(r + x^2) - 2x^2}{(r + x^2)^2} = \frac{r-x^2}{(r+x^2)^2}
\end{equation}
and is only positive on $(-\sqrt{r}, \sqrt{r})$.
There are two reasonable globally positive approximations to take.
The first,
\[
\frac{r}{(r + x^2)^2},
\]
simply ignores the subtracted term $-x^2$. In practice, we found this
approximation to be too aggressive. Instead, we drop the
$2x^2$ from the left of~\eqref{LogHessian} to obtain the approximation
\begin{equation}
\label{LogHessianApprox}
\frac{(r + x^2)}{(r + x^2)^2} = \frac{1}{(r+x^2)}\;.
\end{equation}
Similarly, the terms indexed by superscript $S$ in~\eqref{Ck} and~\eqref{Hk} 
%$H_k$ in~\eqref{TrobH} and~\eqref{TrendH}
provide globally positive definite approximations to the Hessian of $f$,
using the strategy in~\eqref{LogHessianApprox}.
This strategy offers a significant computational advantage---the Hessian approximation that is built up
down-weights the contributions of outliers, helping the algorithm proceed faster to the solution.
As we shall see, these terms are also essential for the general convergence theory.

Our approach exploits the objective structure by iteratively
linearizing $F$ about the iterates $x^k$ and solving the
{\it direction finding subproblem}
\begin{equation}
\label{DirectionFindingSubproblem}
\begin{array}{lll}
&\displaystyle \min_{d\in \mB{R}^{nN}}&
\rho(F(x^k) + F^{(1)}(x^k)d)  + \frac{1}{2}d^TU(x^k)d,
\end{array}
\end{equation}
where $U(x^k)$ is a symmetric positive semidefinite matrix that depends continuously on $x^k$.
For any smoother in the framework of
section~\ref{GenT}, 
problem~\eqref{DirectionFindingSubproblem}
can be solved with a single block-tridiagonal solve of the
system~\eqref{QuadraticSubproblemOne},
yielding descent directions $d$ for the objective $K(x)$. 

We now develop a general convergence theory for convex-composite methods
to establish the overall convergence to a stationary point of $K(x)$.
This theory is in the spirit of~\cite{AravkinL12011} and~\cite{Burke85}, and allows
the inclusion of the quadratic term $\frac{1}{2}d^TU(x^k)d$ in~\eqref{DirectionFindingSubproblem}.
This term was not necessary in~\cite{AravkinL12011}, but is crucial here.
Note that the theory does not rely at all on the technique used to solve
the direction finding subproblem, and so the theory in this paper applies
to the algorithm in~\cite{AravkinL12011} by taking $U = 0$. 

Recall from~\cite{Burke85} that the first-order necessary condition for optimality
in the convex composite problem involving $K(x)$ is
\[
0 \in \partial K(x) = \partial \rho \left( F(x) \right) F^{(1)} (x)
\]
where \(\partial K (x)\) is the
generalized subdifferential of \(K\) at \(x\) \cite{RTRW}
and \(\partial \rho \left( F(x) \right) \) is the convex subdifferential
of \( \rho \) at \(F(x)\) \cite{Rock70}.
Elementary convex analysis gives us the equivalence
\[
0 \in \partial K(x)
\quad \Leftrightarrow \quad
K(x) = \inf_d \rho \left( \; F(x) + F^{(1)} (x) d \; \right) \; .
\]

For the general smoothing class of interest, 
it is desirable to modify this objective by including curvature information,
yielding the problem \eqref{DirectionFindingSubproblem}.
We define the difference function
\begin{equation}
\label{extendedDelta}
\Delta( x; d) = \rho \left( \; F(x) + F^{(1)} (x) d \; \right)
+
\frac{1}{2}d^\R{T}U(x)d
 - K(x) \; ,
\end{equation}
where $U(x)$ is positive semidefinite and varies continuously
with $x$. Note that $\Delta(x;d)$ is a convex function of $d$
that is bounded below, hence the optimal value 
\begin{equation}
\label{extendedInf}
\Delta^* ( x)      = \inf_d \Delta (x ; \; d) \; 
\end{equation}
is well defined regardless of the existence of a solution. 
If \(\Delta^* ( x) = 0\), then 
 \(\displaystyle0 \in \arg\min_d \Delta(x;\;d) \).
Hence, by~\cite[Theorem 3.6]{Burke85},  \(\Delta^* ( x) = 0\) if and only if 
$0 \in \partial K(x)$.

Given \(\eta \in (0,1)\), we define a set of search directions at \(x\) by
\begin{equation}
\label{extendedDSet}
D( x, \eta ) = \set{d}{ \Delta(x ; d) \le \eta \Delta^* ( x) }  \; .
\end{equation}
Note that if there is a \(d \in D( x, \eta)\) such that
\(\Delta( x ; \; d) \ge - \eta \varepsilon \), then
\( \Delta^* ( x) \ge - \varepsilon\).
These ideas motivate the following
algorithm.% \cite{Burke85}.

\begin{algor}
\label{GaussNewtonAlgorithm}
{\it Gauss-Newton Algorithm.}

The inputs to this algorithm are
\begin{itemize}
\item \( x^0 \in \mB{R}^{Nn} \): initial estimate of state sequence
\item \(\varepsilon \ge 0 \): overall termination criterion
\item \( \eta \in (0,1) \): search direction selection parameter
\item \( \beta \in (0, 1) \): step size selection parameter
\item \( \gamma \in (0,1) \): line search step size factor
\end{itemize}

The steps are as follows:

\begin{enumerate}
\item
Set the iteration counter \( \nu = 0 \).
\item (Gauss-Newton Step)
\label{GaussNewtonStep}
Find \(d^\nu \) in the set \( D(x^\nu, \eta ) \)
in \ref{extendedDSet}.
Set \(\Delta_\nu = \Delta( x^\nu ; d^\nu) \)
in \ref{extendedDelta} and
{\it Terminate} if \( \Delta_\nu \ge - \varepsilon \).
\item (Line Search) Set
\[
\begin{array}{lll}
t_\nu &=& \max  \gamma^i \\
     &\text{s.t.}&i \in \{ 0, 1, 2, \cdots \} \; \mbox{ and }
    \\
    &\text{s.t.}&\rho \left( F( x^\nu + \gamma^i d^\nu ) \right)
        \le \rho \left( F( x^\nu ) \right)
                     + \beta \gamma^i \Delta_\nu.
\end{array}
\]
\item (Iterate)
Set \(x^{\nu+1} = x^\nu + t_\nu d^\nu \) and return to Step~\ref{GaussNewtonStep}.
\end{enumerate}
\end{algor}
We now present a general global convergence theorem that covers 
any smoother in section~\ref{GenT}.
%, and in particular includes 
%both T-robust and T-trend smoothers as special cases.
This theorem also generalizes~\cite[Theorem 5.1]{AravkinL12011} to include
positive semidefinite curvature terms in the Gauss-Newton framework.

%%%%%%%%%%%%%%%%%%%%%%%%%%%%%%%%%%%%%%%%%%%%%%%
\begin{theorem}
\label{GlobalConvergenceTheorem}
%%%%%%%%%%%%%%%%%%%%%%%%%%%%%%%%%%%%%%%%%%%%%%%
Define  
\begin{equation}
\label{LevelSetDef}
\Lambda := \{u |  \rho (u)\le K(x^0)\}
\end{equation}
and suppose that there exists a $\tau > 0$ such that
$F^{(1)}$ is bounded and uniformly continuous on the set  
\begin{equation}
\label{snot}
S_0 := \overline{\text{co}}\left(F^{-1}(\Lambda)\right) + \tau\mathbb{B}\;.
\end{equation}
If \(\{x^\nu\}\) is a sequence generated by the
Gauss-Newton Algorithm~\ref{GaussNewtonAlgorithm}
with initial point \(x^0\) and \(\varepsilon =0\), then one of the
following must occur:
\begin{enumerate}
\item[(i)]
The algorithm terminates finitely at a point \(x^\nu\) with
\(0\in\partial K(x^\nu)\).
\item[(ii)]
The sequence \(\| d^\nu \|\) diverges to \(+\infty\).
\item[(iii)]
\(\lim_{\nu\in I}\Delta_\nu = \lim_{\nu \in I} \Delta^*(x^\nu)=0\) for every subsequence  \(I\) for which
the set \( \set{d^\nu}{\nu \in I} \) is bounded.
\end{enumerate}

Moreover, if $\bar x$ is any cluster point
of a subsequence $I \subset \B{Z_+}$
such that the subsequence \(\{d^{\nu}|\nu \in I\}\)
is bounded, then $0 \in \partial K(\bar x)$.
\end{theorem}

\begin{proof}
We will assume that none of (i), (ii), (iii) occur and establish a contradiction. 
Then there is a subsequence $I$ such that 
\[
\sup_{\nu \in I} \|d_\nu\| < \infty 
\quad \text{and}\quad 
\sup_{\nu \in I} \Delta_\nu \leq  \zeta< 0\;.
\] 
Since $K(x^\nu)$ is a decreasing sequence that is bounded below by $0$, 
we know that the differences $K(x^{\nu +1}) - K(x^\nu) \rightarrow 0$.
Therefore, by Step 3) of Algorithm~\ref{GaussNewtonAlgorithm}, 
$\zeta t_\nu \Delta_\nu \rightarrow 0$, which implies that $t_{\nu \in I} \rightarrow 0$. 
Without loss of generality we may assume that $t_\nu \leq 1$ and $t_\nu \|d_\nu\| \leq \gamma\tau$
for all $\nu \in I$. Hence for all $\nu \in I$, 
\[
\begin{aligned}
\|F(x^\nu +  t_\nu \gamma^{-1} d^\nu )  - F(x^\nu) \|
& \leq 
t_\nu \gamma^{-1} \int_0^1 \left\| F'(x^\nu + s t_\nu \gamma^{-1} d^\nu) \right\|\|d^\nu\| ds  \\
&\leq \tau M \;, 
\end{aligned}
\]
where $M$ is a bound on $F'$ over $S_0$. Let $K$ be a Lipschitz constant for $\rho$ over the 
compact set $\Lambda + \tau M\mathbb{B}$. Again by Step 3) of Algorithm~\ref{GaussNewtonAlgorithm}, 
for all $\nu \in I$,  
\[
\begin{aligned}
\beta \gamma^{-1} t_\nu \Delta_\nu 
&\leq 
\rho(F(x^\nu + t_\nu \gamma^{-1} d^\nu )) - \rho(F(x^\nu)) \\
& \leq 
t_\nu \gamma^{-1} \Delta_\nu + K \| F(x^\nu + t_\nu \gamma^{-1} d^\nu ) - F(x^\nu) - t\nu\gamma^{-1}F^{(1)}(x^\nu) d^\nu\| \\
&= 
t_\nu \gamma^{-1} \Delta_\nu 
+ 
t_\nu \gamma^{-1} K \left\| \int_0^1 
\left( F^{(1)}(x^\nu + st_\nu\gamma^{-1} d_\nu) -  F^{(1)}(x^\nu) \right) d^\nu ds \right\|\\
& \leq 
t_\nu \gamma^{-1} \left( \Delta_\nu + K \omega(t_\nu \gamma^{-1}\|d_\nu\|) \|d_\nu\| \right) \;,
\end{aligned}
\]
where $\omega$ is the modulus of continuity of $F'$ on $S_0$.
Rearranging, we obtain 
\[
0 \leq (1-\beta) \Delta_\nu  +  K \omega(t_\nu \gamma^{-1}\|d_\nu\|) \|d_\nu\|\;. 
\]
Taking the limit for $\nu \in I$, we obtain the contradiction $0 \leq (1-\beta)\zeta$.
Hence, $\lim_{\nu \in I} \Delta_\nu = 0$, which implies that $\lim_{\nu \in I} \Delta^*(x^\nu) = 0$, 
since $\Delta_\nu \leq \eta\Delta^*(x^\nu)\leq 0$.
%\newpage
%
%Assertions (i), (ii), and (iii)
%are a restatement of \cite[Theorem 2.4]{Burke85}
%in our context, where
%the sets \(D_\nu\) in \cite[Theorem 2.4]{Burke85} are given by
%\(D_\nu = D( x^\nu , \eta ) \).
%The requirement that \( \rho \) be Lipschitz continuous
%on the set \(\set{(u)}{ \rho (u)\le K(x^0)}\)
%is an immediate consequence of
%the fact that \( \rho \) is coercive on its domain, so this set is compact.
%This completes the proof of (i), (ii), and (iii).
%By compactness, the matrices $U$ are uniformly continuous
%in $x$ on this set.

Finally, suppose that \(\bar x\)
is a cluster point of a sequence \(I \subset \B{Z_+}\)
for which \(\{d^\nu \}\) is bounded.
Without loss of generality, there exists a $\bar d$ 
such that $(x^\nu, d^\nu)_{\nu \in I} \rightarrow (\bar x, \bar d)$.
%
%Since $\bar{x}$ is a cluster point of $\{x^{\nu}\}$,
%we can take a convergent subsequence along which
%\(\{d^\nu\}\) are still bounded.
% and by
%continuity
%$H^{\nu}$ converge to $\bar H = H(\bar x)$.
%By Bolzano-Weierstrass, we can then find a subsequence
%\(J \subset I\) and $\bar d, \bar U \in \mB{R}^{Nn}\times \mB{S}_+^n$
%such that
%\((x^\nu, d^\nu, U^\nu) \rightarrow_J ( \bar{x}, \; \bar{d},\; \bar{U} )\).
For all  \(d \in \mB{R}^{Nn}\),
\[
\begin{array}{lll}
\Delta_\nu
&=&
\rho\left(F(x^\nu)
+
F^{(1)}(x^\nu)d^\nu\right)
+
\frac{1}{2}\|d^\nu\|_{U^{\nu}}^2
-
\rho\left(F(x^\nu)\right)\\
&\leq&
\eta \Delta^*(x^\nu) \\
&\leq&
\eta\left(\rho\left(F(x^\nu)
+
F^{(1)}(x^\nu) d \right)
+
\frac{1}{2}\| d\|_{U^{\nu}}^2
-
\rho\left(F(x^\nu)\right)\right)\;, 
\end{array}
\]
where $U^\nu = U(x^\nu)$.
Taking the limit over $J$ gives
\[
\begin{array}{lll}
0
& = &
\rho\left(F(\bar x)
+
F^{(1)}(\bar x)\bar d\right)
+
\frac{1}{2}\|\bar d\|_{\bar U}^2
-
\rho\left(F(\bar x)\right)\\
&\leq&
\eta\left(\rho\left(F(\bar x)
+
F^{(1)}(\bar x) d \right)
+
\frac{1}{2}\| d\|_{\bar U}^2
-
\rho\left(F(\bar x)\right)\right)\;,
\end{array}
\]
where $\overline{U} = U(\bar x)$. 
Since $d$ was chosen arbitrarily, it must be the case that $\Delta^*(\bar x) = 0$, 
which implies that $0 \in \partial K$ by~\cite[Theorem 3.6]{Burke85}.
%But $d$ was an arbitrary point in
%\(\mB{R}^{Nn}\), so in particular we must have
%$\Delta^*(\bar x) = 0$, which %by~\cite[Theorem 3.6]{Burke85}
%implies that $0 \in \partial K$.
\end{proof}

A stronger convergence
result is possible under stronger assumptions on \(F\) and \(F^{(1)}\).
%Fix \(x^0 \in \mB{R}^{Nn} \), and define
%\begin{equation}
%\label{LevelSetDef}
%\Lambda = \set{u}{ \rho (u) \le K(x^0)}\;.
%\end{equation}

%%%%%%%%%%%%%%%%%%%%%%%%%%%%%%%%%%%%%%%%%%%%%%%
\begin{corollary}
\label{GlobalConvergenceCorollary}
%%%%%%%%%%%%%%%%%%%%%%%%%%%%%%%%%%%%%%%%%%%%%%%
Suppose that \(F^{-1}(\Lambda) = \set{x}{F(x)\in\Lambda}\)
is bounded, and there exists $0 < \lambda_{\min}$ such that 
\begin{align}
\label{Assumption}
\forall \; x \in F^{-1}(\Lambda), \quad
0 < \lambda_{\min} \|d\|^2 \leq d^TU(x)d 
\quad \forall d \in \mathrm{Null}(F^{(1)}(x))\;.
\end{align}
If \(\{ x^\nu \}\) is a sequence generated by Algorithm~\ref{GaussNewtonAlgorithm}
with initial point \(x^0\) and \(\varepsilon = 0\), then \(\{ x^\nu \}\) and
\(\{d^\nu\}\) are bounded and either
the algorithm terminates finitely at a point \( x^\nu \) with
\( 0 \in \partial K( x^\nu ) \), or \( \Delta_\nu \rightarrow 0 \)
as \( \nu \rightarrow \infty \),
and every cluster point
\(\bar{x}\) of the sequence \(\{x^\nu\}\)
satisfies \(0 \in \partial K( \bar{x} )\).
\end{corollary}

\begin{proof}
First note that \(F^{-1}(\Lambda)\) is closed since \(F\) is continuous,
and therefore \(F^{-1}(\Lambda)\) is compact, since by assumption it is bounded.
 Hence \(S_0\) (see~\eqref{snot}) is also
 compact. Therefore, \(F^{(1)}\) is uniformly continuous and bounded on
 \(S_0 \) which implies
 that the hypotheses of Theorem \ref{GlobalConvergenceTheorem} are
 satisfied, and so one of (i)-(iii) must hold. If (i) holds we are done, so we
 will assume that the sequence  \(\{x^\nu\}\) is infinite.  Since \(\{ x^\nu \} \subset F^{-1}(\Lambda)\), this sequence is bounded.
 We now show that the sequence \(\{d^\nu\}\) of search directions is also bounded.
% SASHA: stopped here. 

 Suppose that \eqref{Assumption} holds. 
 For any direction $d^{\nu}$, note that $d^\nu$ satisfies
\begin{equation}
\label{SolutionBound}
\rho\left(F(x^\nu) + F^{(1)}(x^\nu)d^{\nu}\right)
+ \frac{1}{2}\|d^{\nu}\|_{U^\nu}^2 \leq \rho\left(F(x^\nu)\right)
\leq
\rho(F(x^0))\;.
\end{equation}
%
%since we can achieve $\rho\left(F(x)\right)$ with $d = 0$.
 Since $\rho\geq 0$, we have
 \begin{equation}
 \label{linearization}
\{F(x^\nu)\} \subset \Lambda \quad \text{and} \quad \{F(x^\nu) + F^{(1)}(x^\nu)d^{\nu}\} \subset \Lambda
 \end{equation}
 and 
 \begin{equation}
 \label{bdDir}
\left\{ 
\frac{1}{2}(d^{\nu})^TU^\nu d^\nu
\right\} 
\leq \rho\left(F(x^0)\right) \quad \forall \nu\;.
\end{equation}
 Suppose that the $\{d^\nu\}$ is unbounded. 
 Then without loss of generality, there exists a subsequence $I$,
 a unit vector $u$, and a vector $\bar x\in F^{-1}(\Lambda)$ 
 such that $\lim_{\nu \in I} d^\nu/\|d^\nu\|\rightarrow u$
 and $\lim_{\nu \in I} x^\nu \rightarrow \bar x$.
 Since $\Lambda$ is bounded,~\eqref{linearization}
 implies that $F^{(1)}(\bar x)u = 0$, so $u \in \mathrm{Nul}(F^{(1)}(\bar x))$,
 and therefore 
 \[
 0 < \lambda_{\min} \leq u^TU(\bar x)u\;.
 \]
 On the other hand, by~\eqref{bdDir}, $\frac{1}{2}\left(\frac{d^\nu}{\|d^\nu\|}\right)^TU^\nu 
 \left(\frac{d^\nu}{\|d^\nu\|}\right)
 \leq \frac{\rho (F(x^0))}{\|d^\nu\|^2}$
 and so in the limit we have the contradiction
 \[
 0 < \lambda_{\min} \leq u^TU(\bar x)u \leq 0\;.
 \]
 Hence $d^\nu$ are bounded. The result now follows from Theorem~\ref{GlobalConvergenceTheorem}.
 
 \end{proof}
We now show that all smoothers of section~\ref{GenT} 
satisfy the required assumptions of Theorem~\ref{GlobalConvergenceTheorem}
and Corollary~\eqref{GlobalConvergenceCorollary}.

\begin{corollary}[Smoother Satisfaction]
\label{SmootherSatisfaction}
Suppose that the process and measurement functions $g_k$ and $h_k$ 
in~\eqref{NonlinearStudentModel} are twice continuously differentiable.  
Then for $F$ given in~\eqref{FDef}, $F^{(1)}$ is bounded and uniformly 
continuous on $S_0$ in~\eqref{snot}. 
Moreover, the hypotheses of Corollary~\ref{GlobalConvergenceCorollary} hold 
if for all $x$ in $F^{-1}(\Lambda)$ and for all $k$, 
there exists $\eta$ such that 
\[
0 < \eta < \sigma_{\min}(G^S(x)), \quad 
G^S(x)
:= 
\begin{bmatrix}
I & 0 & 0 & \\
-(g_k^{(1)}(x))^S & I  & 0 & 0 \\
0 & \ddots & \ddots& \ddots & \\
& 0 & -(g_k^{(N-1)}(x))^S & I
\end{bmatrix} 
\]
\end{corollary}

\begin{proof}
We first show that both $\Lambda$ and $F^{-1}(\Lambda)$ are bounded. 
The first claim follows immediately by the coercivity of $\rho$ in~\eqref{rhoDef}. 
To verify the second claim, we will show that for any sequence of 
$x^\nu$ with $\|x^\nu\| \rightarrow \infty$, 
we can find a subsequence  $J$ such that $\lim_{\nu \in J} \|w^\nu\| = \infty$, which implies
the existence of subsequence $I$ such that either $\lim_{\nu \in I} \|w^G\| = \infty$ or 
$\lim_{\nu \in I} f(x^\nu) = \infty$. 
In particular there does not exist an unbounded 
sequence $\{x^\nu\}$ with $F(x^\nu) \subset \Lambda$ , and 
therefore $F^{-1}(\Lambda)$ must be bounded. 

If $\|x^\nu\|\rightarrow \infty$, we can find an index 
$k \subset [1, \dots, N]$ and subsequence $J$ such that $\lim_{\nu \in J}\| x^\nu_k\|= \infty$. 
Now, either $\lim_{\nu \in J} w^\nu_k = \infty$ and we are done, 
or $\lim_{\nu \in J}\|g_k(x^\nu_{k-1})\| = \infty$, so $\lim_{\nu \in J} 
\|x^\nu_{k-1}\| = \infty$. 
Iterating this argument, we arrive at the limiting case $w^\nu_1 = x^\nu_1 - x^0_1$, 
and so if all $\|w^\nu_j\|$ are bounded for $j > 1$, we can guarantee that $\lim_{\nu \in J}
\|w^\nu_1\|= \infty$.

Since $F$ is twice continuously differentiable by the hypotheses on $g$ and $h$, 
the boundedness of $F^{-1}(\Lambda)$ establishes the boundedness and uniform
continuity of $F^{(1)}$ on $S_0$ in~\eqref{snot} for any $\tau > 0$.

It remains to show that condition~\eqref{Assumption} is satisfied. 
Let $\Sc{W}^G$, $\Sc{W}^S$ denote the indices
associated to all subvectors $w_k^G$ and $w_k^S$
within $w_k$. If $d \in \mathrm{Null}(F^{(1)}(x))$, then necessarily 
$d_{\Sc{W}^G} = 0$. This is simply because $F^{(1)}$ is nonsingular 
on $\Sc{W}^G$, since it contains the sub matrix 
\[
G^G(x) :=
\begin{bmatrix}
    \R{I}  & 0      &          &
    \\
    -(g_2^{(1)})^{G}(x_{1})   & \R{I}  & \ddots   &
    \\
        & \ddots &  \ddots  & 0
    \\
        &        &   -(g_N^{(1)})^G(x_{N-1})   & \R{I}
\end{bmatrix}\;,
\]
which is the standard process matrix $G$ projected to those 
coordinates where Gaussian modeling is applied. 
To finish the analysis, we present the full form of the matrix $U$ 
restricted to $\Sc{W}^S$:

\begin{equation}
\label{hessianApproxU}
U
=
\begin{bmatrix}
U_1 & A_2^\R{T} & 0 & \\
A_2 & U_2  & A_3^\R{T} & 0 \\
0 & \ddots & \ddots& \ddots & \\
& 0 & A_N & U_N
\end{bmatrix} 
+ \R{diag}(\{H_k\}
,
\end{equation}
where
\begin{eqnarray}
\nonumber
&A_k 
\label{AkS}
&=
 -\frac{r (Q_k^S)^{-1}(g^{(1)}_k)^S }{r + \|w_k^S\|_{(Q_k^S)^{-1}}^2}\\
\nonumber
&U_k 
\nonumber
&= 
\label{UkS}
\frac{r((g^{(1)}_{k+1})^S)^\R{T}(Q^S_{k+1})^{-1}(g^{(1)}_{k+1})^S}
{{r + \|w_{k+1}^S\|_{(Q_{k+1}^S)^{-1}}^2}}
+  \frac{r (Q_k^S)^{-1}}{r + \|w_k^S\|_{(Q_k^S)^{-1}}^2} \\
\nonumber
&H_k
&=
\frac{s((h_k^{(1)})^S)^\R{T}(R_k^S)^{-1}(h_k^{(1)})^S}
{(s + \|v_k^S\|_{(R_k^S)^{-1}}^2)} \\
\label{HkS}
\end{eqnarray}

Note that we can write the first summand in~\eqref{hessianApproxU} as
\(
(G^S)^T \widetilde Q^{-1} G^S\;,
\)
where 
\[
\begin{aligned}
\quad 
\widetilde Q^{-1} &:= \R{diag}(\{\widetilde Q_k^{-1}\}),
\quad
\widetilde Q_k^{-1} &=  \frac{r (Q_k^S)^{-1}}{r + \|w_k^S\|_{(Q_k^S)^{-1}}^2}\;.
\end{aligned}
\]
Since $F^{-1}(\Lambda)$ is bounded, the denominators of $\widetilde Q_k^{-1}$ are bounded, 
and so eigenvalues of $\widetilde Q_k^{-1}$ are bounded from below, and the singular 
values of $G^S$ are bounded from above. 

We now have
\[
0 < \eta_{\min} \leq \sigma_{\min} (G^S) \leq \sigma_{\max} (G^S) \leq \eta_{\max}
\] 
for all $x$,  where the upper found follows from Theorem~\cite[2.2]{AravkinBellBurkePillonetto2013} together with compactness of $F^{-1}(\Lambda)$. 

Then, by~\cite[Theorem 2.1]{AravkinBellBurkePillonetto2013}, we have
\[
\kappa((G^S)^T \widetilde Q^{-1} G^S) \leq 
\frac{\lambda_{\max}(\widetilde Q^{-1})\eta^2_{\max}}
{\lambda_{\min}(\widetilde Q^{-1}) \eta_{\min}^2}\;.
\]
for all $x\in F^{-1}(\Lambda)$.
This completes the proof. 
\end{proof}

\begin{remark}
One can also consider conditions on the individual $g_k^S$ that can produce a lower bound $\eta$ on $G^S$,
as required by Corollary~\ref{SmootherSatisfaction}. One such condition is
\begin{equation}
\label{lowerBound}
0< \eta \leq \left\{ 1 + \sigma^2_{\min}(g_{k+1}^{(1)}) - \sigma_{\max}(g_k^{(1)}) - \sigma_{\max}(g_{k+1}^{(1)}) \right\}
\end{equation}
If this condition is satisfied, then by Theorem~\cite[2.2]{AravkinBellBurkePillonetto2013}, $\eta < \sigma_{\min}(G^S)$. 
However, this condition is sufficient, and may not be necessary. 
\end{remark}

\section{Numerical Experiments}
\label{SimulationRobust}
%

%%%%%%%%%%%%%%%%%%%%%%%%%%%%%%%%%%%%%%
\subsection{T-Robust Smoother: function reconstruction using splines}
\label{sec:LinearExample}
%%%%%%%%%%%%%%%%%%%%%%%%%%%%%%%%%%%%%%

In this section we compare the new T-robust smoother with the
$\ell_2$-Kalman smoother~\cite{Bell2008} and with the
$\ell_1$-Laplace robust smoother~\cite{AravkinL12011}, both
implemented in~\cite{ckbs}.
The {\it ground truth}
for this simulated example is
\[
x(t) = \begin{bmatrix} -\cos(t) & -\sin(t)\end{bmatrix}^\R{T} \; .
\]
The time between measurements is a constant $\Delta t$.
We model the two components of the state as the first and second integrals of 
white noise, so that
$$
g_k ( x_{k-1} ) =
\begin{bmatrix}
    1        & 0
    \\
    \Delta t & 1
\end{bmatrix} x_{k-1}
\; ,
\qquad
Q_k =
\begin{bmatrix}
    \Delta t   & \Delta t^2 / 2
    \\
    \Delta t^2 / 2 & \Delta t^3 / 3
\end{bmatrix}\;.
$$
This stochastic model for function reconstruction underlies the Bayesian interpretation
of cubic smoothing splines, see \cite{Wahba1990} for details.\\
The measurement model for the conditional mean of measurement $z_k$ given state $x_k$ is
defined by
$$
h_k( x_k ) =  \begin{bmatrix} 0 & 1 \end{bmatrix} x_k = x_{2,k} \;,
\qquad R_k = \sigma^2\;,
$$
where \( x_{2,k} \) denotes the second component of \( x_k \),
$\sigma^2 = 0.25$ for all experiments, and the degrees of
freedom parameter $k$ was set to 4 for the Student's t methods.

The measurements \( \{z_k \} \) were generated
as a sample from
\[
z_k = x_2 ( t_k ) + v_k, \quad t_k=0.04\pi \times k
\]
where  $k=1,2,\ldots,100$. The measurement noise $v_k$ was
generated according to the following schemes.
\begin{enumerate}
\item {\bf Nominal}:
\(v_k \sim \B{N} ( 0 , 0.25 ).\)\\
\item {\bf Gaussian contamination}
\[v_k \sim (1 - p ) \B{N} (0, 0.25 ) + p \B{N} ( 0 , \phi ),\]
for $p \in \{ 0.1 , 0.2, 0.5 \}$
and $\phi \in \{1, 4, 10, 100 \}$.\\
\item  {\bf Uniform contamination}
\[v_k \sim (1 - p ) \B{N} (0, 0.25 ) + p \B{U} (-10, 10 ),\]
for $p \in \{0.1, 0.2, 0.5\}$.
\end{enumerate}

Each experiment was performed 1000 times.
Table~\ref{SimulationResults} presents
the results for our simulated fitting showing the
median Mean Squared Error (MSE) value and a quantile interval containing
95\% of the results. The MSE is defined by
\begin{equation}
\label{MSEeq}
    \frac{1}{N} \sum_{k=1}^N
        [ x_1 ( t_k ) - \hat{x}_{1,k} ]^2
        +
        [ x_2 ( t_k ) - \hat{x}_{2,k} ]^2 ,
\end{equation}
where \( \{ \hat{x}_k \} \) is the corresponding estimating sequence.

From Table~\ref{SimulationResults} one can see that
T-Robust and
the $\ell_1$-smoother perform as well as the (optimal)
$\ell_2$-smoother at nominal conditions, and that both continue to
perform at that same level for a variety of outlier generating
scenarios. T-Robust always performs at least as well as the
$\ell_1$-smoother, and it gains an advantage when either the
probability of contamination is high, or the contamination is
uniform. This is likely due to the re-descending influence function
of the Student's t-distribution --- the smoother effectively throws
out bad points rather than simply decreasing their impact to a
certain threshold, as is the case for the $\ell_1$-smoother.
As an example, results coming from a single run
for the case where 50\% of measurements are contaminated
with the uniform distribution on $[-10, 10]$ are displayed in Figure \ref{simComp}.
Notice that
T-Robust has an advantage over
the $\ell_1$-smoother.
\begin{table}
\small \caption{ \label{SimulationResults} Function reconstruction
via spline: median MSE over 1000 runs and intervals containing 95\%
of MSE results.}
\begin{center}
\begin{tabular}{|c|c|c|c|c|}\hline
{\bf Outlier }
    & {p}
    & {\B {$\ell_2$ MSE}}
    & {\B {$\ell_1$ MSE}}
    & {\B {Student's t MSE}}
\\ \hline
Nominal
 & ---
    &.04(.02, .1)
    &.04(.01, .1)
    & .04(.01, .09)
\\ \hline
$\B{N}(0, 1)$
& .1
    &.06(.02, .12)
    &.04(.02, .10)
    &.04(.02, .10)
\\ \hline
$\B{N}(0, 4)$
& .1
  &.09(.04, .29)
    &.05(.02, .12)
    &.04(.02, .11)
\\ \hline
$\B{N}(0, 10)$
& .1
    &.17(.05, .55)
    &.05(.02, .13)
    &.04(.02, .11)
\\ \hline
$\B{N}(0, 100)$
& .1
    &1.3(.30, 5.0)
    &.05(.02, .14)
    &.04(.02, .11)
\\ \hline
$\B{U}(-10, 10)$
& .1
    &.47(.12, 1.5)
    &.05(.02, .13)
    &.04(.02, .10)
\\ \hline
$\B{N}(0, 10)$
& .2
    &.32(.11, .95)
    &.06(.02, .19)
    &.05(.02, .16)
\\ \hline
$\B{N}(0, 100)$
& .2
    &2.9(.94, 8.5)
    &.07(.02, .22)
    &.05(.02, .14)
\\ \hline
$\B{U}(-10, 10)$
& .2
    &1.1(.36, 3.0)
    &.07(.03, .26)
    &.05(.02, .13)
\\ \hline
$\B{N}(0, 10)$
& .5
   &.74(.29, 1.9)
    &.13(.05, .49)
    &.10(.04, .45)
\\ \hline
$\B{N}(0, 100)$
& .5
    &7.7(2.9, 18)
    &.21(.06, 1.6)
    &.09(.03, .44)
\\ \hline
$\B{U}(-10, 10)$
& .5
    &2.6(1.0, 5.8)
    &.20(.06, 1.4)
    &.10(.03, .44)
\\ \hline
%$\B{L}(.25) $
%  %  &.86(.52, 1.4)
%    &.15(.06, .38)
%    &.07(.03, .21)
%    &.07(.03, .21)
%\\\hline
%$\B{T}(0, 2, df = 4)$
%  %  &.86(.44, 1.7)
%    &.12(.04, .40)
%    &.10(.03, .30)
%    &.10(.03, .30)
%\\ \hline
\end{tabular}
\end{center}
\end{table}

\begin{figure*}
\begin{center}
{\includegraphics[scale=0.45]{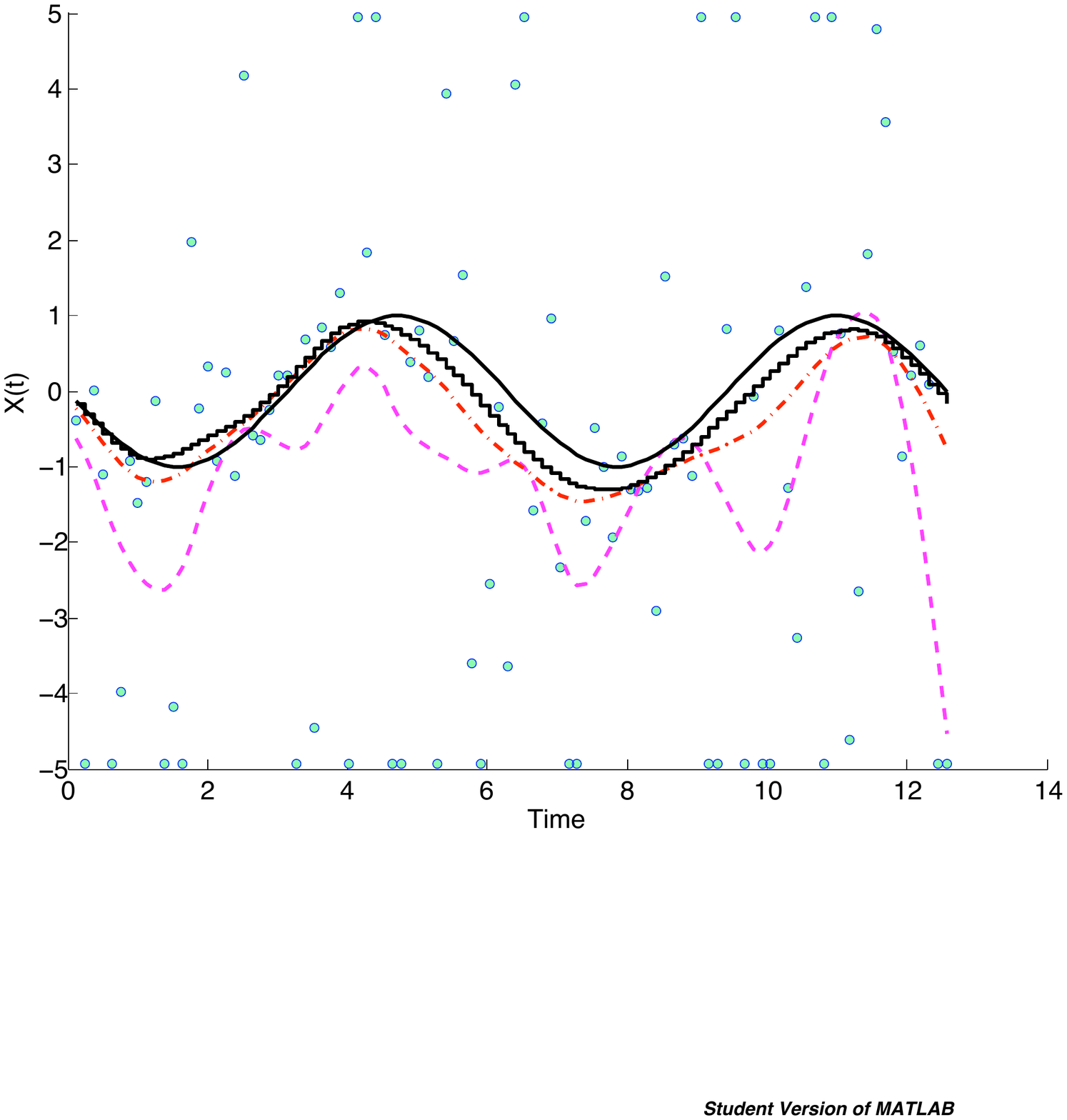}}
\caption{\label{simComp} Function reconstruction via spline: performance of $\ell_2$ Kalman smoother
(dash), $\ell_1$-Laplace Robust smoother (dash-dot), and T-Robust
(stair-case solid) on contaminated normal model with 50\% outliers
distributed uniformly on [-10,10]. True state x(t) is drawn as solid
line. Measurements appear as `o' symbols, and all measurements
visible off of the true state are outliers in this case. Values
outside [-5,5] are plotted on the axis limits. }
\end{center}
\end{figure*}

%%%%%%%%%%%%%%%%%%%%%%%%%%%%%%%%%%%%%%%%%%%%%%%%%%%%%%%
\subsection{T-Robust Smoother: Van Der Pol oscillator}
\label{NumexpNonlinear}
%%%%%%%%%%%%%%%%%%%%%%%%%%%%%%%%%%%%%%%%%%%%%%%%%%%%%%%

\begin{figure}
\hspace{-.4in}  {\includegraphics[scale=0.57]{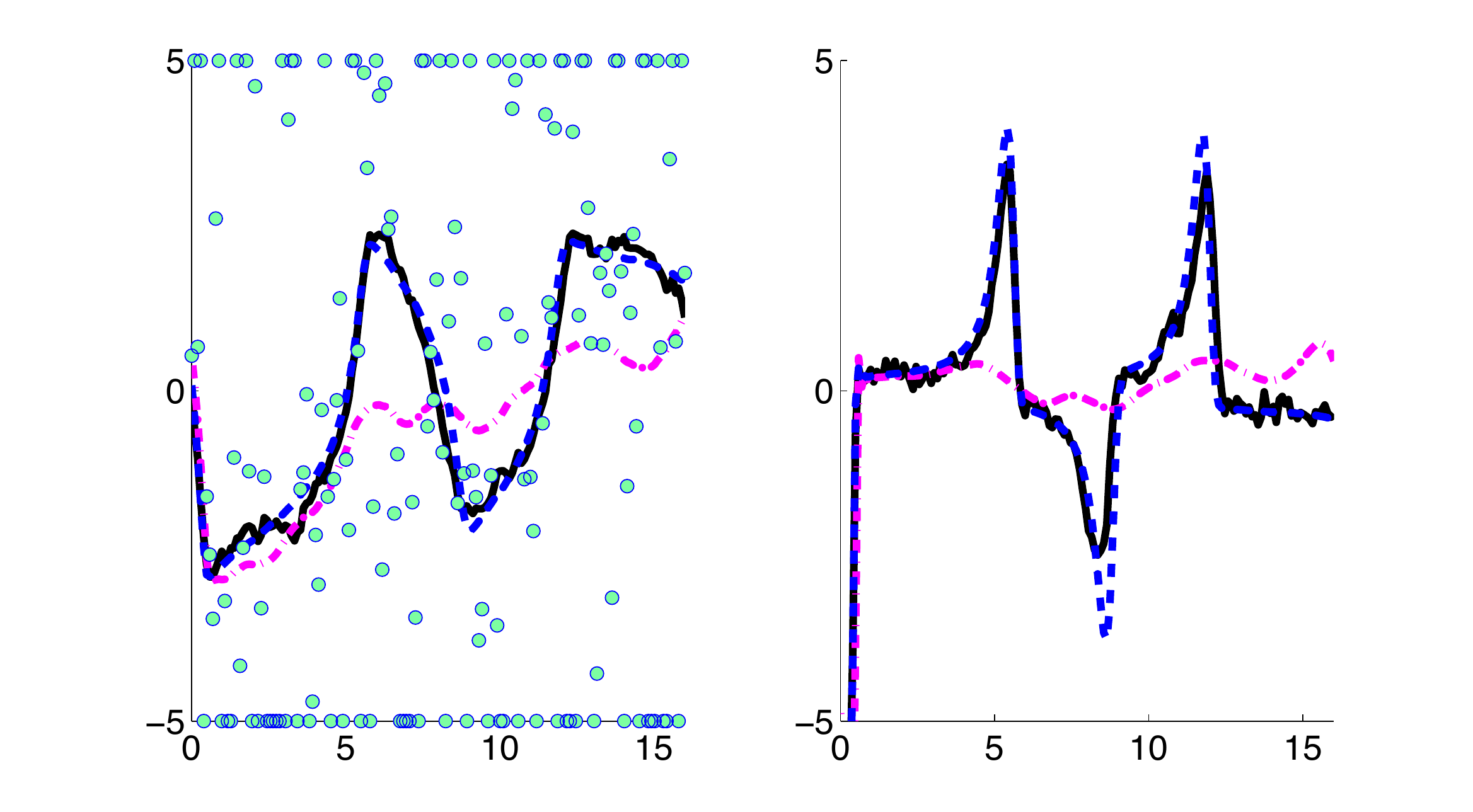}}
\caption{Van Der Pol oscillator: smoother fits for X-component (left) and Y-component (right), with
{\bf 70\% outliers $N(0, 100)$}.
Black solid line is truth, magenta dash-dot is the $\ell_1$
smoother result, and blue dashed line is T-robust. Measurements on X-component
are shown as dots, with outliers outside the range $[-5, 5]$ plotted on top and bottom axes. \label{VanDerPolFig}}
\end{figure}

In this section, we present results for the
Van Der Pol oscillator (VDP), described in detail
in ~\cite{AravkinL12011}. The VDP oscillator is a
coupled nonlinear ODE defined by
\begin{eqnarray*}
\dot{x}_1 (t) &=& x_2 (t)\\
\dot{x}_2 (t) &=& \mu [ 1 - x_1 (t)^2 ] x_2 (t) - x_1 (t)
\end{eqnarray*}
The process model here is the Euler approximation for $X(t_k)$
given $X(t_{k-1})$:
\[
g_k ( x_{k-1} ) = \left( \begin{array}{cc}
    x_{1,k-1} + x_{2,k-1} \Delta t
    \\
    x_{2,k-1} + \{ \mu [ 1 - x_{1,k}^2 ] x_{2,k} - x_{1,k} \} \Delta t
\end{array} \right) \; .
\]
For this simulation, the {\it ground truth} is obtained from
a stochastic Euler approximation
of the VDP.
To be specific,
with \( \mu = 2 \), \( N = 164 \) and \( \Delta t = 16 / N \),
the ground truth state vector \( x_k \) at time \( t_k = k \Delta t \)
is given by \( x_0 = ( 0 , -0.5 )^\R{T} \) and
for \( k = 1, \ldots , N \),
$x_k = g_k ( x_{k-1} ) + w_k$,
where $\{ w_k \}$ is a realization of
independent Gaussian noise with variance $0.01$.

In~\cite{AravkinL12011}, the $\ell_1$-Laplace smoother was shown to have superior
performance to the $\ell_2$-smoother, both implemented in~\cite{ckbs}.
We compared the performance of the nonlinear T-robust and nonlinear
 $\ell_1$-Laplace smoothers, and found that T-robust gains
 an advantage in the extreme cases of 70\% outliers.
 Figure~\ref{VanDerPolFig} illustrates
 results coming from a single representative run.
 For 40\% or fewer outliers, it is hard 
 to differentiate the performance of the two smoothers for this 
 nonlinear example.  

%%%%%%%%%%%%%%%%%%%%%%%%%%%%%%%%%%%%%%%%%%%%%%%%%%%%%%%
\subsection{T-Robust Smoother: underwater Tracking Application}
%%%%%%%%%%%%%%%%%%%%%%%%%%%%%%%%%%%%%%%%%%%%%%%%%%%%%%%

This application is described in detail in \cite{AravkinL12011},
so we just give a brief overview here.
In \cite{AravkinL12011} we used the application to test the
$\ell_1$-Laplace smoother. Here we use it for a qualitative
comparison between the T-Robust smoother,
the $\ell_1$-Laplace smoother, and the $\ell_2$ smoother with
outlier removal.

In this experiment, a tracking target was hung on a steel cable
approximately 200 meters below a ship. The pilot was attempting
to keep the ship in place (hold station) at specific coordinates,
but the ship was pitching and rolling due to wave action. The
measurements for the smoother were sound travel times
between the tracking target and four bottom
mounted transponders at known locations, and pressure readings
from a gauge that was placed on the target. Tracking data
was independently verified using a GPS antenna mounted on a ship,
and the GPS system provided sub-meter accuracy in position.

Pressure measurements in absolute bars were converted to depth in meters
by the formula
\[ \R{depth} = 9.9184 ( \R{pressure} - 1 ). \]
We use $N$ to denote the total number of time points at which
we have tracking data.
For $k = 1, \dots, N$,
the state vector at time $t_k$ is defined by
\(
x_k = ( e_k , n_k , d_k , \dot{e}_k , \dot{n}_k , \dot{d}_k )^\R{T}
\)
where $( e_k , n_k , d_k )$ is the ( east, north, depth )
location of the object (in meters from the origin),
and $( \dot{e}_k , \dot{n}_k , \dot{d}_k )$
is the time derivative of this location. 

The measurement vector at time $t_k$ is denoted by $z_k$.
The first four components of $z_k$ are the range measurements to the
corresponding bottom mounted transponders and the last component is the depth
corresponding to the pressure measurement.
For \( j = 1 , \ldots , 4 \),
the model for the mean of the corresponding range measurements was
\[
h_{j,k} ( x_k ) = \| ( e_k , n_k , d_k ) - b_j \|_2 - \Delta r_j \; .
\]
These measurements were assumed
independent with standard deviation $3$ meters.
These depth measurements were assumed to have standard deviation of
0.05 meters.\\
We use \( \Delta t_k \) to denote \(  t_{k+1} - t_k \).
The model for the mean of \( x_{k+1} \) given \( x_k \) was
\[
\begin{array}{lll}
g_{k+1} ( x_k ) = \\
(
    e_k + \dot{e}_k \Delta t_k , \;
    n_k + \dot{n}_k \Delta t_k , \;
    d_k + \dot{d}_k \Delta t_k , \;
    \dot{e}_k , \;
    \dot{n}_k , \;
    \dot{d}_k
)^\R{T} \; .
\end{array}
\]
\begin{figure}
\begin{center}
{\includegraphics[scale=0.38]{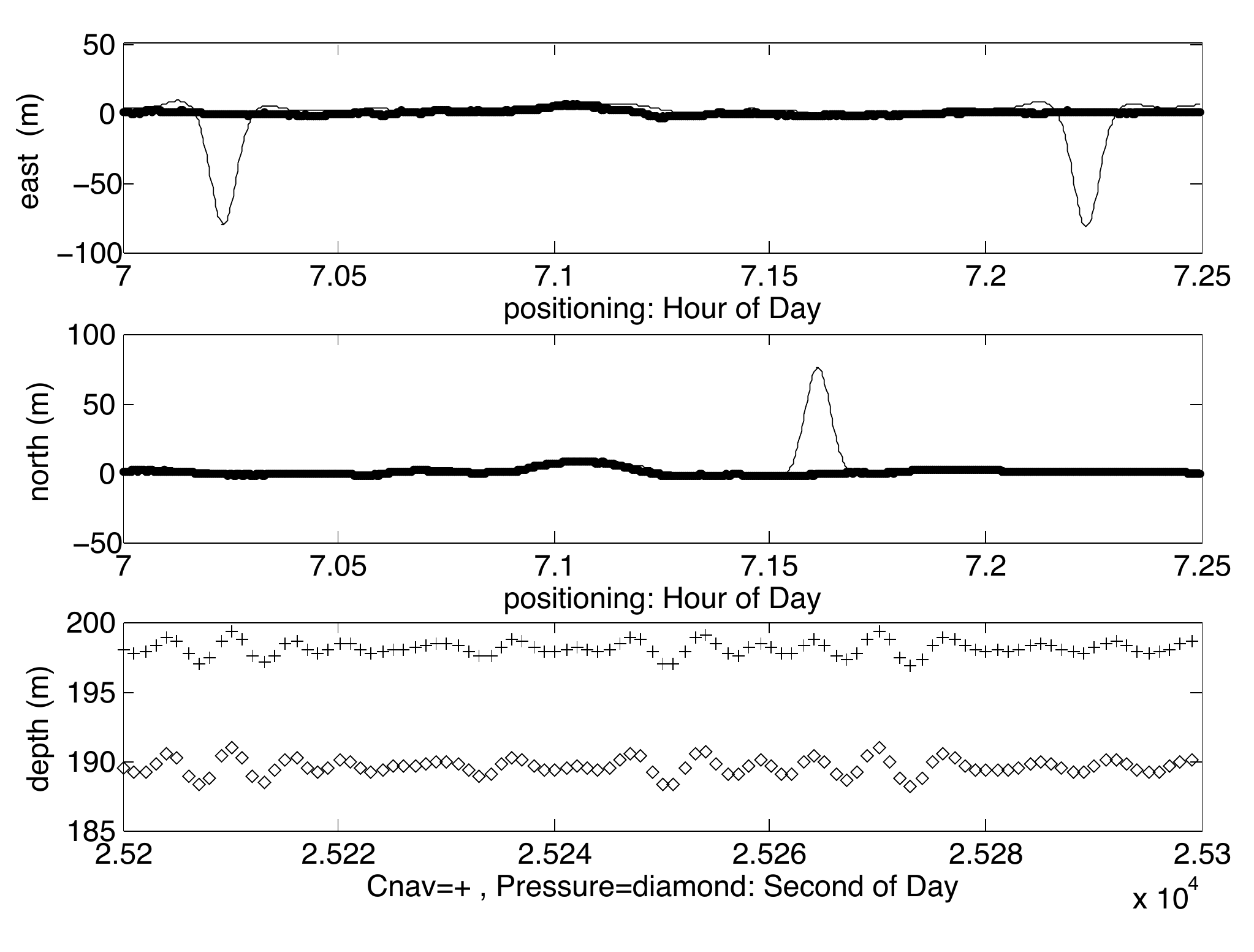}}
\end{center}
\caption{
    \label{GaussianTrack}
    Track: Independent GPS verification (thick line and +),
    $\ell_2$-smoother estimate (thin line). Note
        the large outliers in the data.
}
\end{figure}
%q
\begin{figure}
\begin{center}
{\includegraphics[scale=0.26]{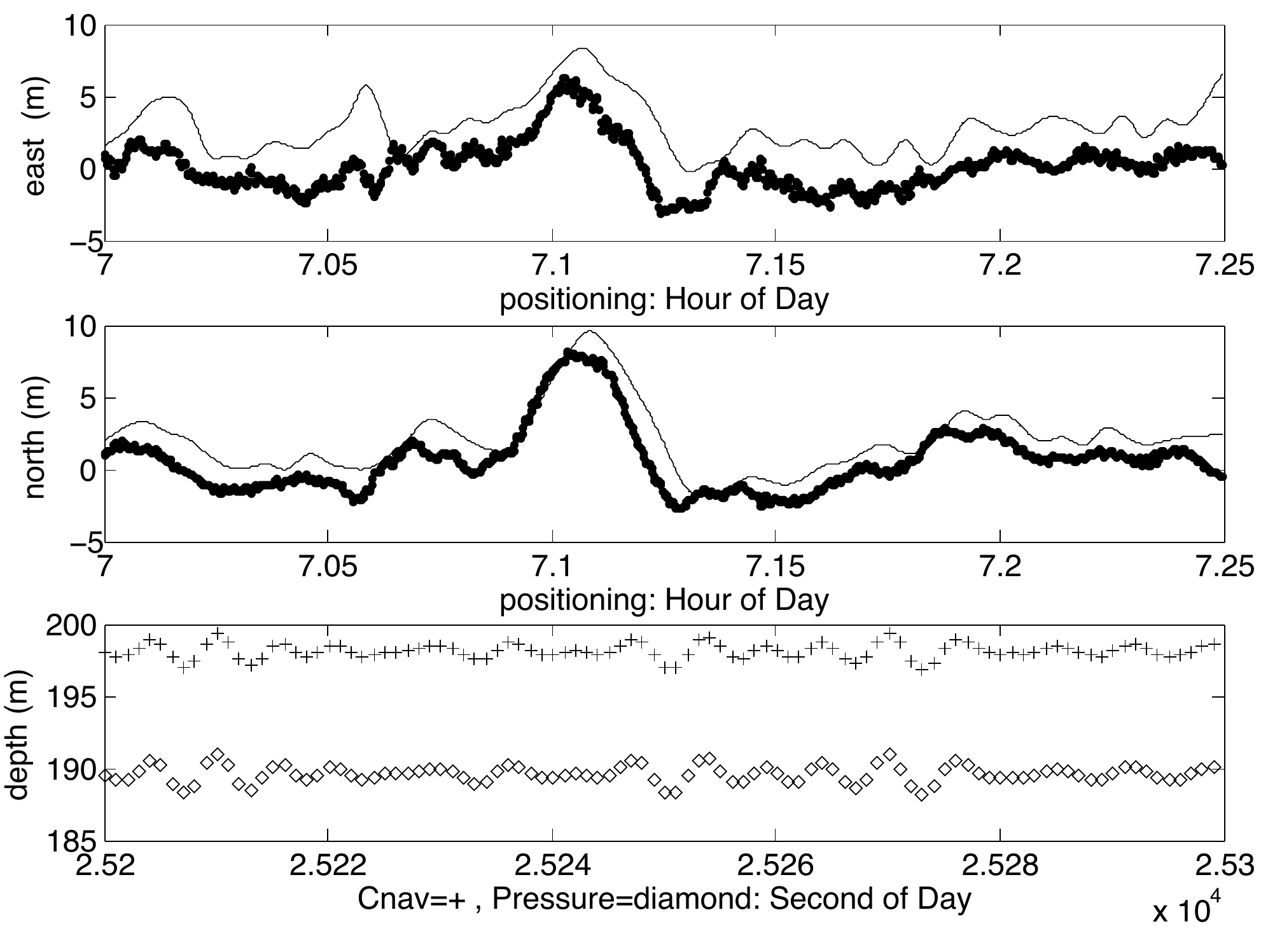}}
($a_1$)
{\includegraphics[scale=0.26]{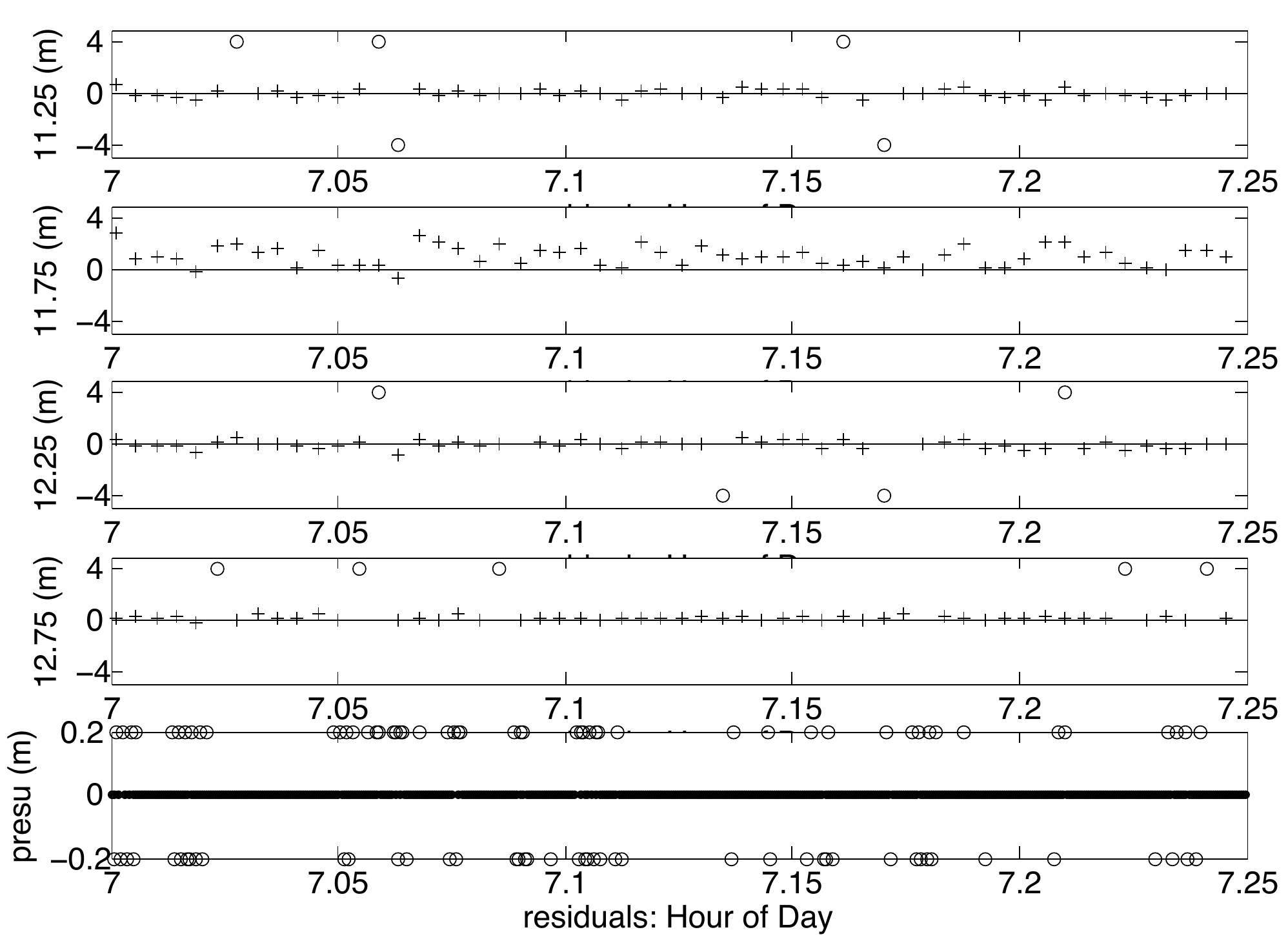}}
($a_2$)
{\includegraphics[scale=0.25]{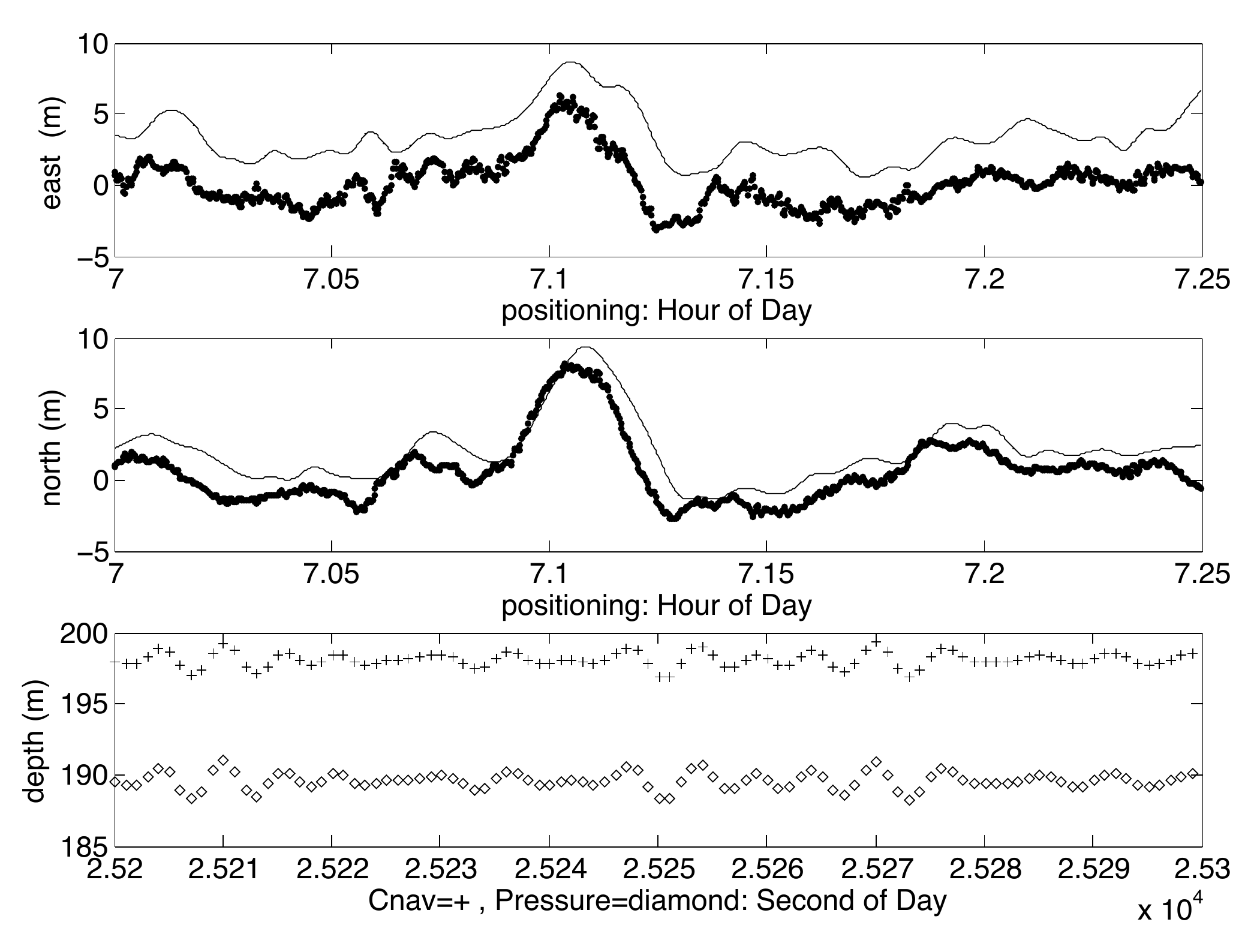}}
($b_1$)
{\includegraphics[scale=0.25]{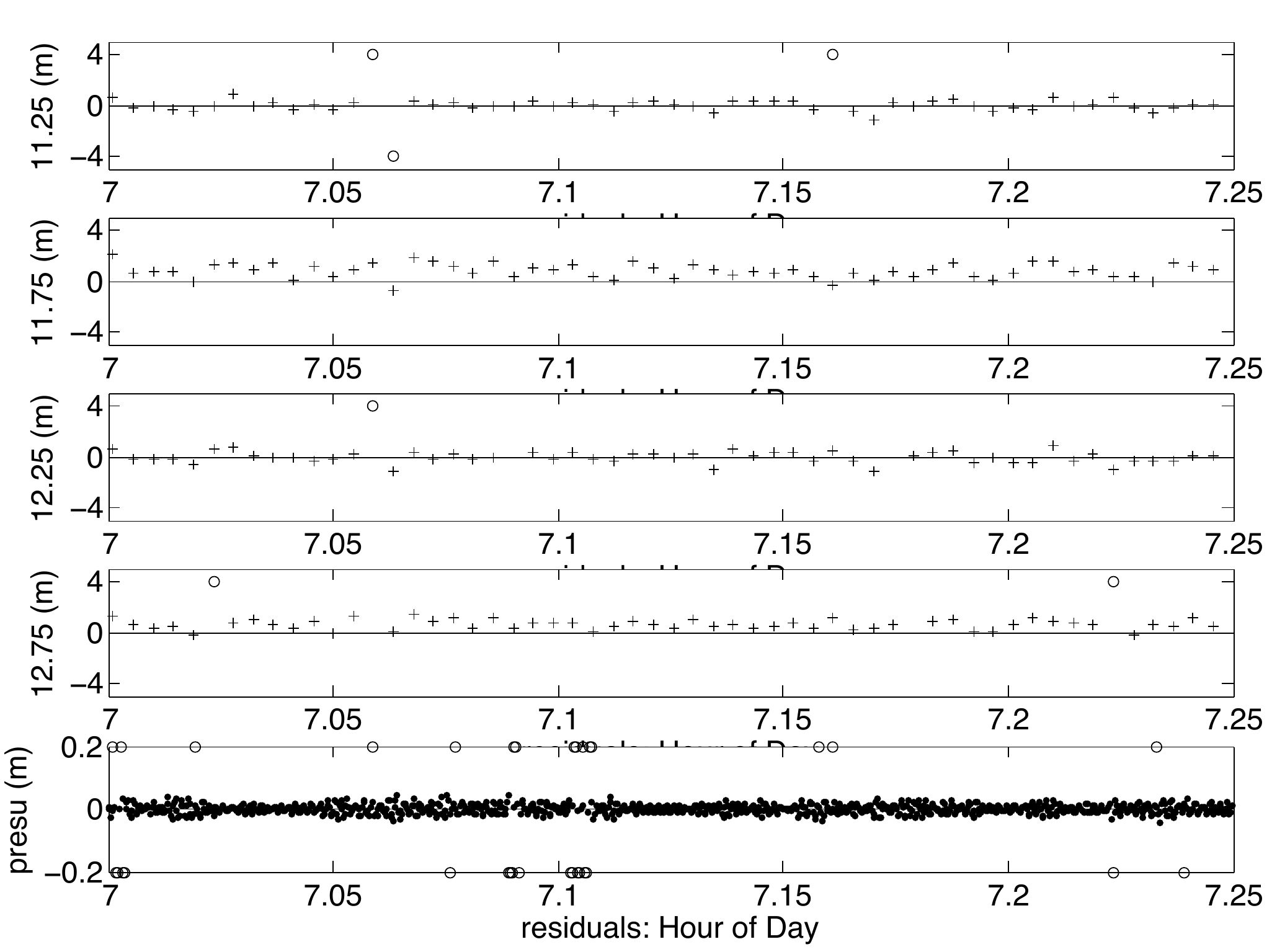}}
($b_2$)
{\includegraphics[scale=0.26]{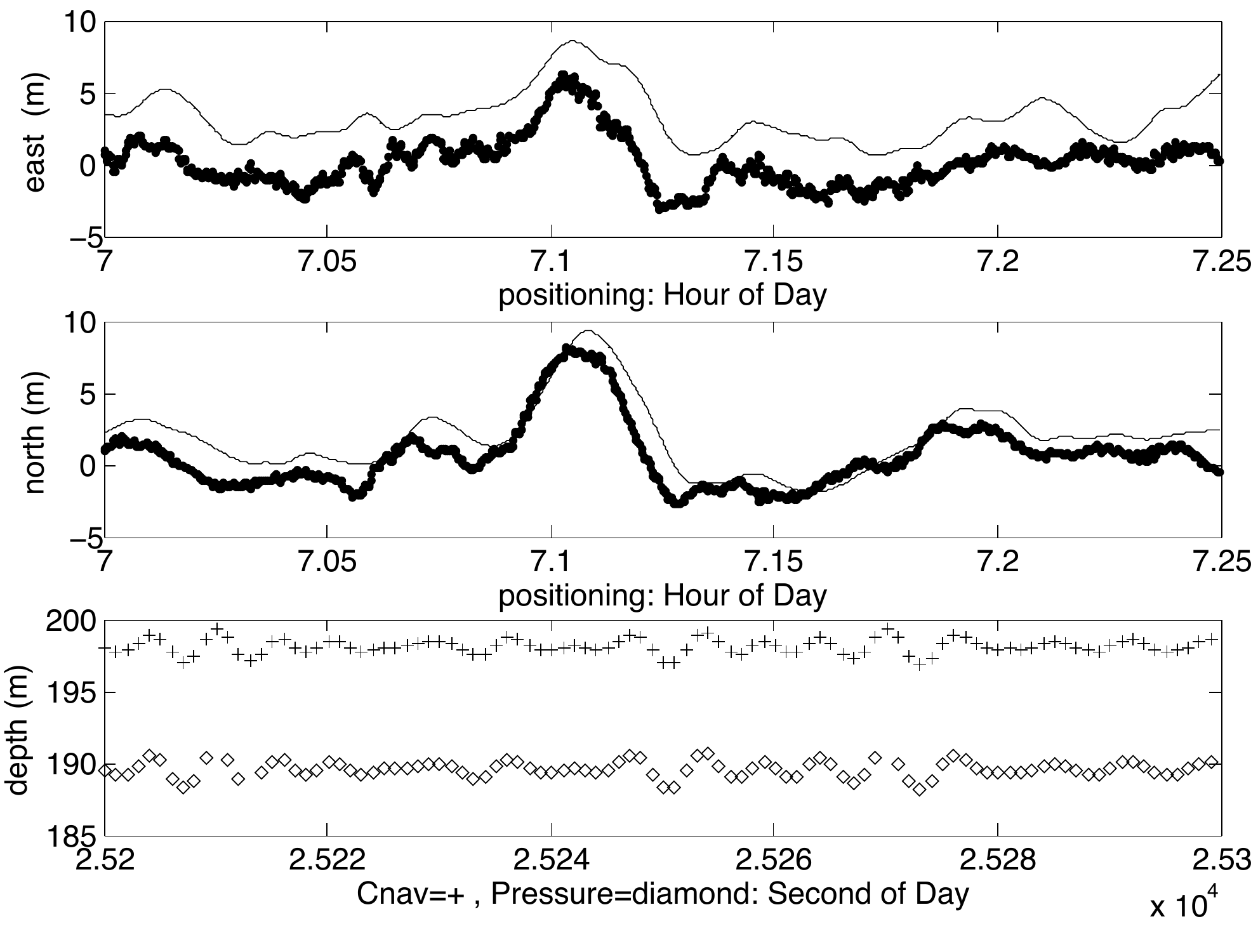}}
($c_1$)
{\includegraphics[scale=0.26]{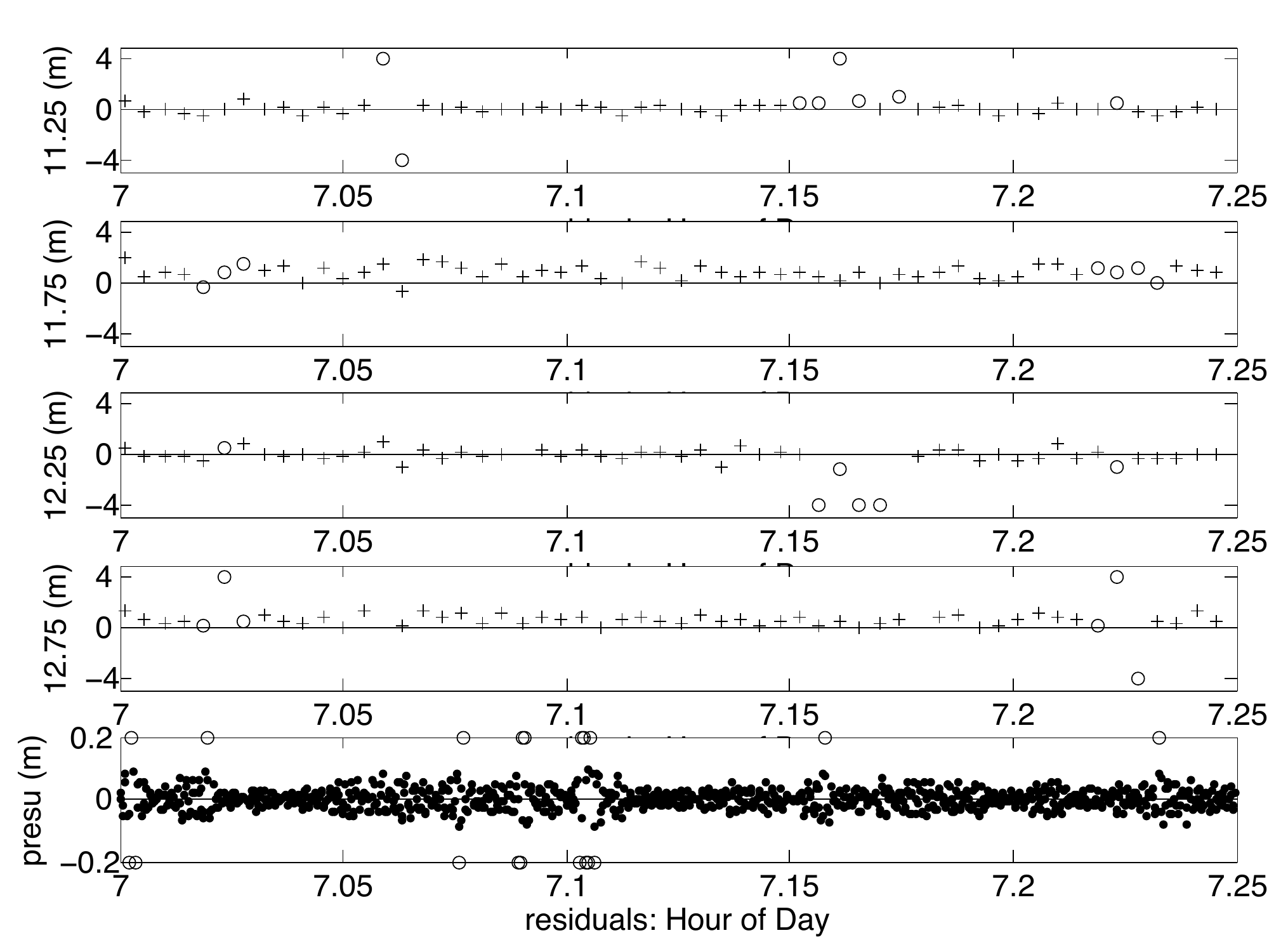}}
($c_2$)
\end{center}
\caption{
    \label{CompareTrack}
    Track: Independent GPS verification (thick line and +)
    and Residuals for 
    (a): $\ell_1$-Laplace smoother (thin line),
    (b): T-Robust smoother (thin line),
        (c): $\ell_2$-smoother with outlier removal.
}
\end{figure}
The process noise corresponding to east, north, and depth components of the conditional distribution of \( x_{k+1} \) given \( x_k \)
was assumed to be Gaussian, with mean zero and standard deviation
$ .01 \Delta t_k$.
The process noise corresponding to the derivative vector of
east, north, and depth components
of the conditional mean \( x_{k+1} \) given \( x_k \)
was also assumed Gaussian with mean zero
and standard deviation $ .2 \Delta t_k$.

$\ell_2$-smoother results without outlier removal are shown in Figure~\ref{GaussianTrack}.
There are three large peaks (two in the east component and
one in the north component of the state) that are due to
measurement outliers, and require either an outlier removal
strategy or robust smoothing.

Three fits are shown in Figure~\ref{CompareTrack}: $\ell_1$-Laplace,
T-Robust, and $\ell_2$-smoother with outlier removal.
The darker curves appearing below the track are independent verifications
using the GPS tracking near the top of the cable.
A depth of 198 meters was added to the depth location of the GPS antenna
so that the depth comparison can use the same axis for both the GPS
data and the tracking results.
Note that the time scale for the depth plots different (much finer) than
the north, east, down plots,
and demonstrates the accuracy of the GPS tracking as validated by the
pressure sensor.\\
 T-Robust, like the $\ell_1$-Laplace smoother, was able to use the whole data sequence, despite large outliers in the data.
The fits look very similar, and it is clear that T-Robust
can also be used for smoothing in the presence of outliers.
Note that the T-Robust track (b) is smoother than the
$\ell_1$-Laplace track (a) but has more detail than the
$\ell_2$-smoother track with outlier removal (c). This is
easiest to see by comparing the east coordinates
in (a), (b), and (c) of Figure ~\ref{CompareTrack},
between 7.2 and 7.25 hours.

The residual plots in Figure~\ref{CompareTrack} are quite revealing.
Outliers are defined as measurements corresponding to residuals
with absolute value greater than three standard deviations from the mean.
All outliers
are shown as `o' characters, and those that
fall outside the axis limits are plotted on the vertical axis
limit lines.
Note that the $\ell_2$-smoother with outlier removal
detects outliers after the first fit
that are not outliers after the second fit.
The peaks in Figure~\ref{GaussianTrack} are large enough to
influence the entire fit, and so some points which are actually
`good' measurements are removed by the 3-$\sigma$ edit rule,
resulting in `over-smoothing' of the outlier removal track
and more detail in both of the robust smoothers in Figure~\ref{CompareTrack}.

The $\ell_1$-Laplace smoother pushes more of the residuals to zero,
particularly those corresponding to depth measurements,
which are the most reliable and frequent.
The T-Robust smoother is somewhere in between ---
the residuals for the depth track are smaller in comparison
to the residuals of the $\ell_2$-smoother, but are not
set to zero as by the $\ell_1$-Laplace smoother.
As discussed previously, these features are artifacts of
the behavior of the distributions at zero, and
the choice of smoother should be guided by particular
applications.

%%%%%%%%%%%%%%%%%%%%%%%%%%%%%%%%%%%%%%%%%%%%%%%%%%%%%%%
\subsection{T-Trend Smoother: reconstruction of a sudden change in state}
\label{NumexpT-Trend}
%%%%%%%%%%%%%%%%%%%%%%%%%%%%%%%%%%%%%%%%%%%%%%%%%%%%%%%

%
We present a proof of concept result for the T-Trend smoother,
using two Monte Carlo studies of 200 runs.
In the first study, the state vector, as well as the process
and measurement models, are the same as in Sec.~\ref{sec:LinearExample}. 
At any run,  $x_2$ has to be reconstructed
from 20 measurements corrupted by a white Gaussian noise of variance 0.05 and
collected on $[0,2\pi]$ using a uniform sampling grid.
The top panel of Figure \ref{jumpSmoother1} reports the boxplot of the 200 root-MSE errors
%, with MSE defined by
%$$
 %   \sqrt {\frac{1}{N} \sum_{k=1}^N
 %       [ x_2 ( t_k ) - \hat{x}_{2,k} ]^2}
%$$
%calculated using 
for the $\ell_2$-, $\ell_1$-, and T-Trend 
smoothers, while the top right panel of Figure \ref{jumpSmoother1}
displays the estimate obtained in a single run. It is apparent that
the performance of the three estimators
is very similar.

The second experiment is identical to the first one except that we
introduce a `jump' at the middle of the sinusoidal wave. The bottom
panel of Figure \ref{jumpSmoother1} reveals the superior performance
of the T-Trend smoother under these perturbed conditions. The result
depicted in the bottom right panel of Figure \ref{jumpSmoother1} for a single 
run of the experiment is representative of the average performance of
the estimators. The estimate achieved by the $\ell_2$-smoother
(dashed-line) does not follow the jump well (the true state is the
solid line). The $\ell_1$-smoother (dashdot) does a better job than
the $\ell_2$-smoother, and the T-trend smoother outperforms the
$\ell_1$-smoother, following the jump very closely while
still providing a good solution along the rest of the path.

\begin{figure*}
\begin{center}
  {\includegraphics[scale=0.25,angle=-90]{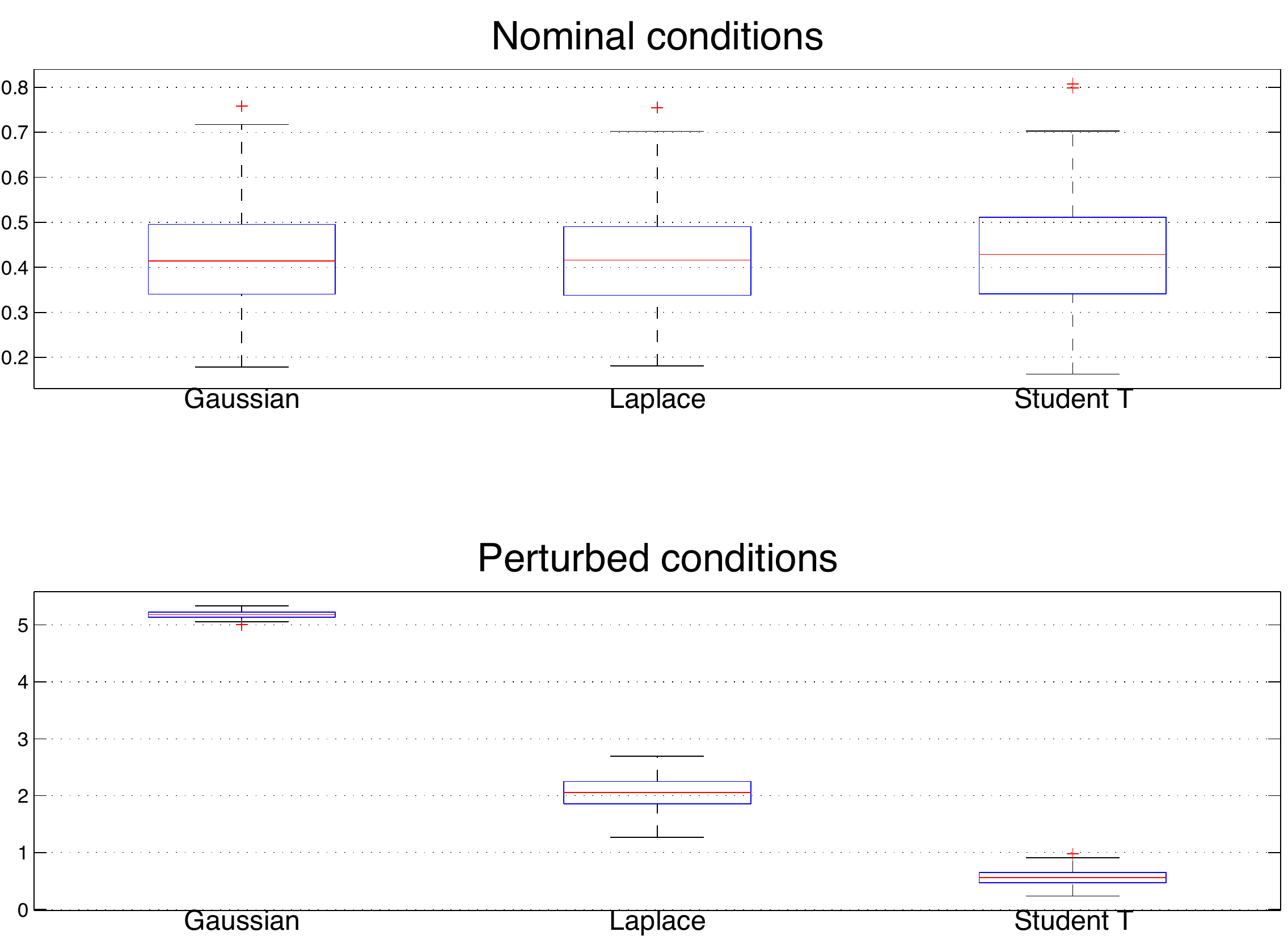}}
    {\includegraphics[scale=0.25,angle=-90]{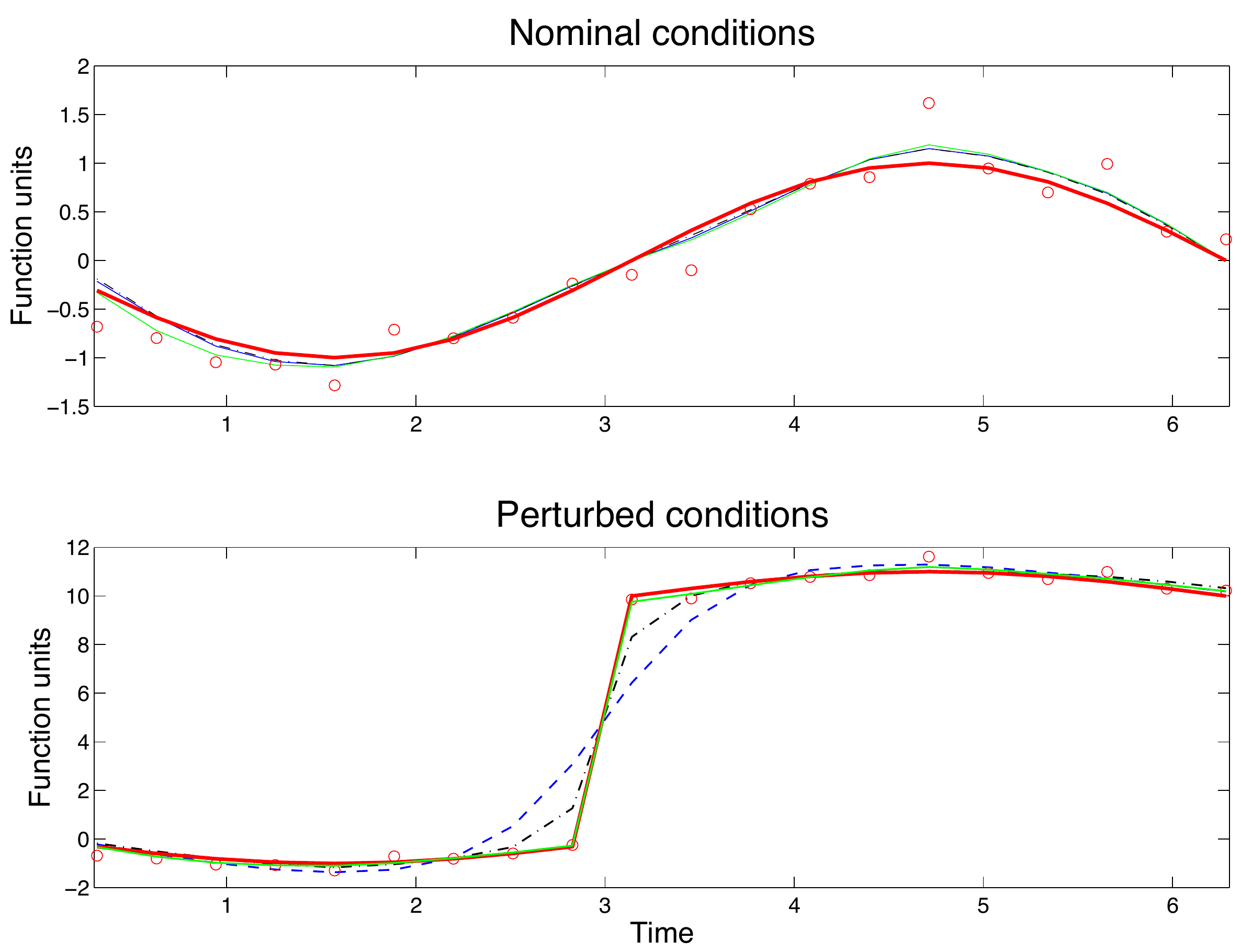}}
\end{center}
\caption{Reconstruction of a sudden change in state obtained by $\ell_2$, 
$\ell_1$, and T-Trend smoothers. 
{\bf Left:} Boxplot of reconstruction errors under nominal (top) and perturbed
(bottom) conditions.  {\bf Right}: Reconstructions obtained 
using $\ell_2$ (dashed), $\ell_1$ (dashdot) and T-Trend
(thin line) smoother. The thick line is the true state.
\label{jumpSmoother1} }
\end{figure*}
%
%\begin{figure*}
%\begin{center}
%  {\includegraphics[scale=0.42,angle=-90]{TFigures/FigNvsP2.pdf}}
%\end{center}
%\caption{Reconstruction of a sudden change in state: results from a
%Monte Carlo run under nominal (top) and perturbed (bottom)
%conditions using $\ell_2$ (dashed), $\ell_1$ (dashdot) and $T-Trend$
%(thin line) smoother. The thick line is the true state.
%\label{jumpSmoother2} }
%\end{figure*}
%

%%%%%%%%%%%%%%%%%%%%%%%%%%%%%%%%%%%%%%%%%%%%%%
\subsection{Reconstruction of a sudden change in state in the presence of outliers}
\label{NumexpTT}
%%%%%%%%%%%%%%%%%%%%%%%%%%%%%%%%%%%%%%%%%%%%%%

Until now, we have considered robust and trend applications 
separately, in order to compare with previous robust smoothing formulations 
and to highlight the main features of the trend-filtering problem. 
A natural extension is to consider these features in tandem --- in other words, 
can we smooth a track which has {\it both} outliers and a sudden change in state? 
In fact, smoothers of this nature (but exploiting convex formulations) 
have already been proposed~\cite{Farahmand2011}. 

The challenge to building such a strong smoother  
is that without prior knowledge, it is difficult to tell the difference 
between a bad measurement (an outlier)
and a good measurement that may be consistent with a sudden change in the state.  
In many cases, the user will be aware that some sensors are reliable, while others are subject 
to contamination. This kind of prior information can now easily be incorporated
 using the generality and flexibility of section~\ref{GenT}, so that 
the user may specify trustworthy sensors (by modeling corresponding
residuals {\it indices} with Gaussians) as well as stable state components 
(by modeling corresponding innovation residual {\it indices} with Gaussians). 
Note that this is very different from specifying which of the individual {\it measurements} are reliable,
or which individual {\it transitions} follow the process model.

In this section, we consider a situation where we have a trustworthy sensor 
$s_1$ and an occasionally malfunctioning sensor $s_2$. Sensor $s_2$
gives frequent measurements, but some proportion of the time is subject 
to heavy contamination, while sensor $s_1$ gives measurements rarely,
but they are trustworthy (i.e. only subject to small Gaussian noise).  
Using the flexible interface implemented in~\cite{ckbs}, we can 
model $s_1$ errors as Gaussian and $s_2$ errors as Student's t. 

We use setup in section~\ref{NumexpT-Trend} together with the Gaussian 
outlier contamination scheme described in section~\ref{sec:LinearExample}. 
Both measurements are direct, so the measurement matrix in this case is 
$$
H_k x_k  =  \begin{bmatrix} 0 & 1 \\ 0 & 1\end{bmatrix} x_k \;,
\qquad R_k = \begin{bmatrix} \sigma^2  \\ & \sigma^2\end{bmatrix}\;.
$$
Since in the ckbs interface, the user specifies $R_k^{-1}$ rather than $R_k$, 
missing measurements are easily specified by setting the corresponding component of 
$R_k^{-1}$ to $0$. 

For the contaminated sensor $s_2$, we consider $p = .2$ contamination level, and 
$\phi = 100$, very large contaminating variance. We have $s_2$ measurements at 
every time step, but $s_1$ measurements only at every $10$th time step.

The results are shown in figure~\ref{DoubleT}. 
    Measurements are plotted using diamonds, 
   with $s_2$ measurements represented by small symbols, while 
	$s_1$ measurements are represented by large symbols.
   Ground truth is shown using a sold black line, and smoother results are shown
   using a red dashed line. 
Results in panel (a)  were obtained using
the least squares smoother, which cannot handle outliers. 
Results in panel (b) were obtained using T-Robust only, applying
Student's t modeling only to the measurement components. The resulting fit 
is much better, but the smoother struggles to follow the jump in the track,
overestimating the curve before the jump and under-estimating it after the jump.
Results in panel (c) were obtained by the Double T smoother, 
which modeled all residuals and innovations using Student's t. 
Double T follows the curve well before the jump, but not after. 
Note that there are a couple of measurements sitting along the sharp
jump --- the Double T suspects these to be outliers, only trusting the 
concentrated measurements to the right of the jump. 

Finally, results in panel (d) were obtained by using the information about 
which measurements are reliable. Specifically, Student's t modeling was 
used for all innovations residuals and for $s_2$, and Gaussian modeling 
was used for the $s_1$ component. This smoother ignores the outliers
and is able to follow the jump very well, since it takes advantage of the fact
that there is a reliable measurement that happens to be sitting right
in the middle of the transition.  

The file used to generate the subplots in the figure is 
\verb{noisy_jump_two_meas.m{, 
which can be accessed through the \verb{example{ subdirectory
of~\cite{ckbs}.

\begin{figure*}
\begin{center}
{\includegraphics[scale=0.33]{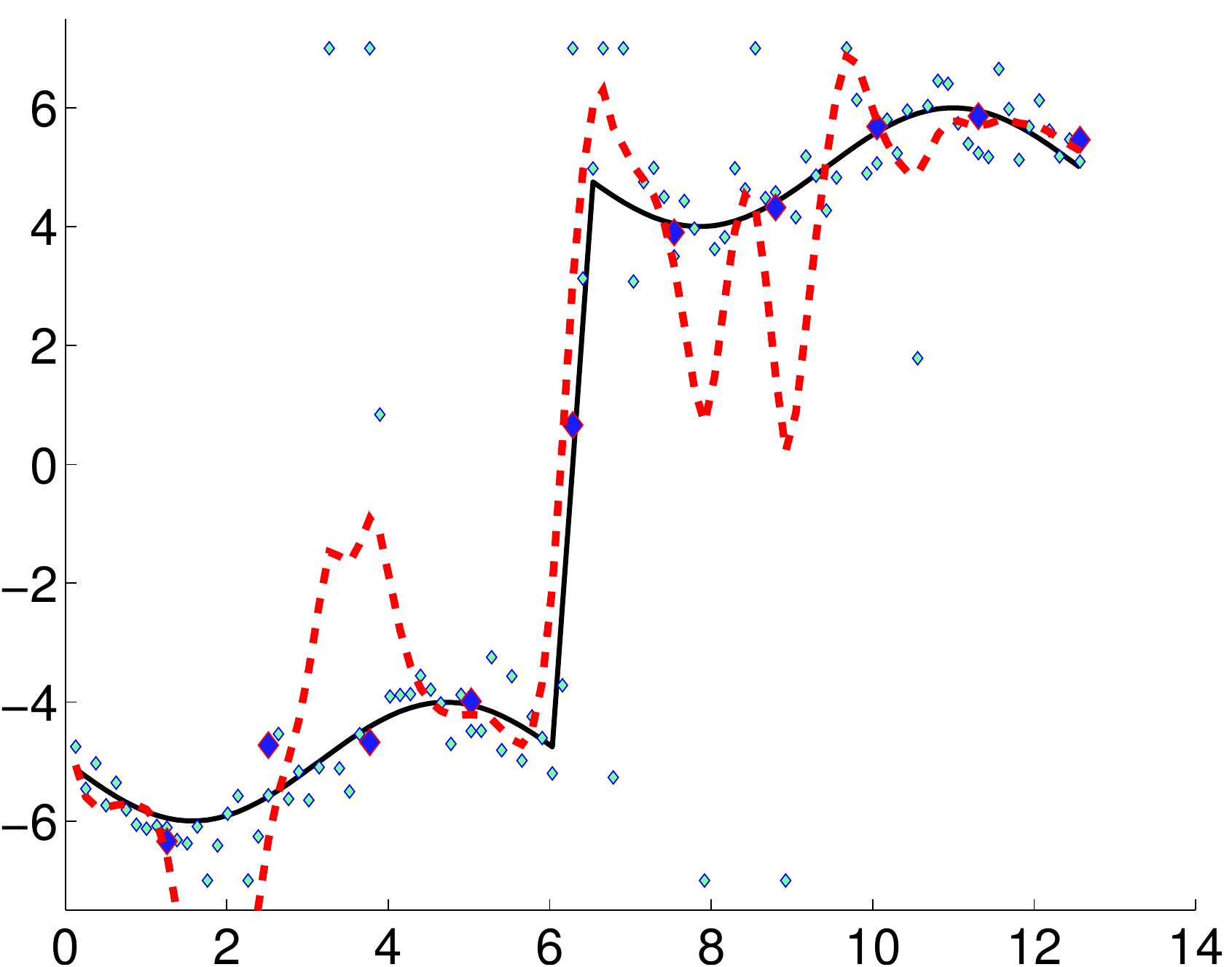}}
(a)
{\includegraphics[scale=0.33]{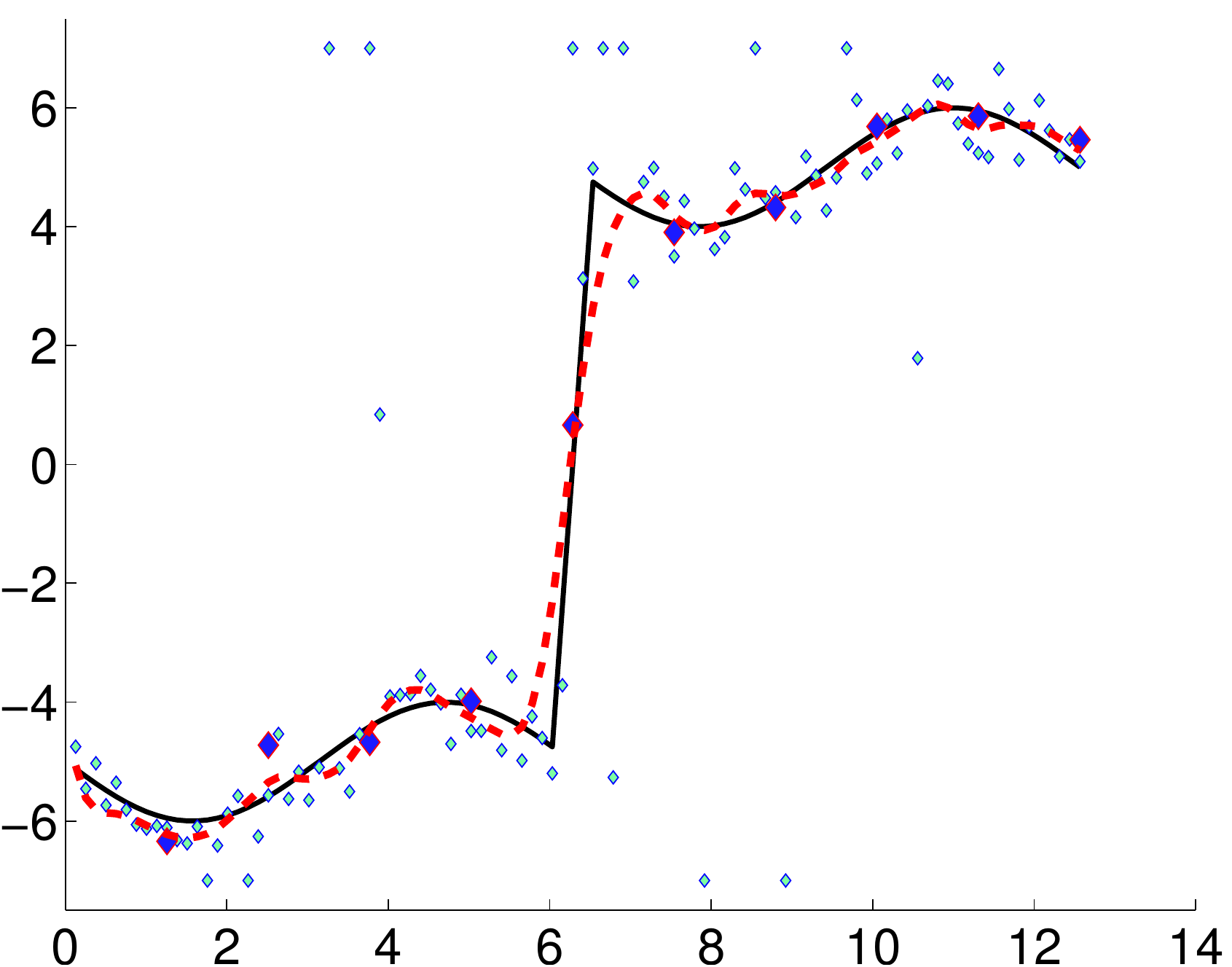}}
(b)
{\includegraphics[scale=0.33]{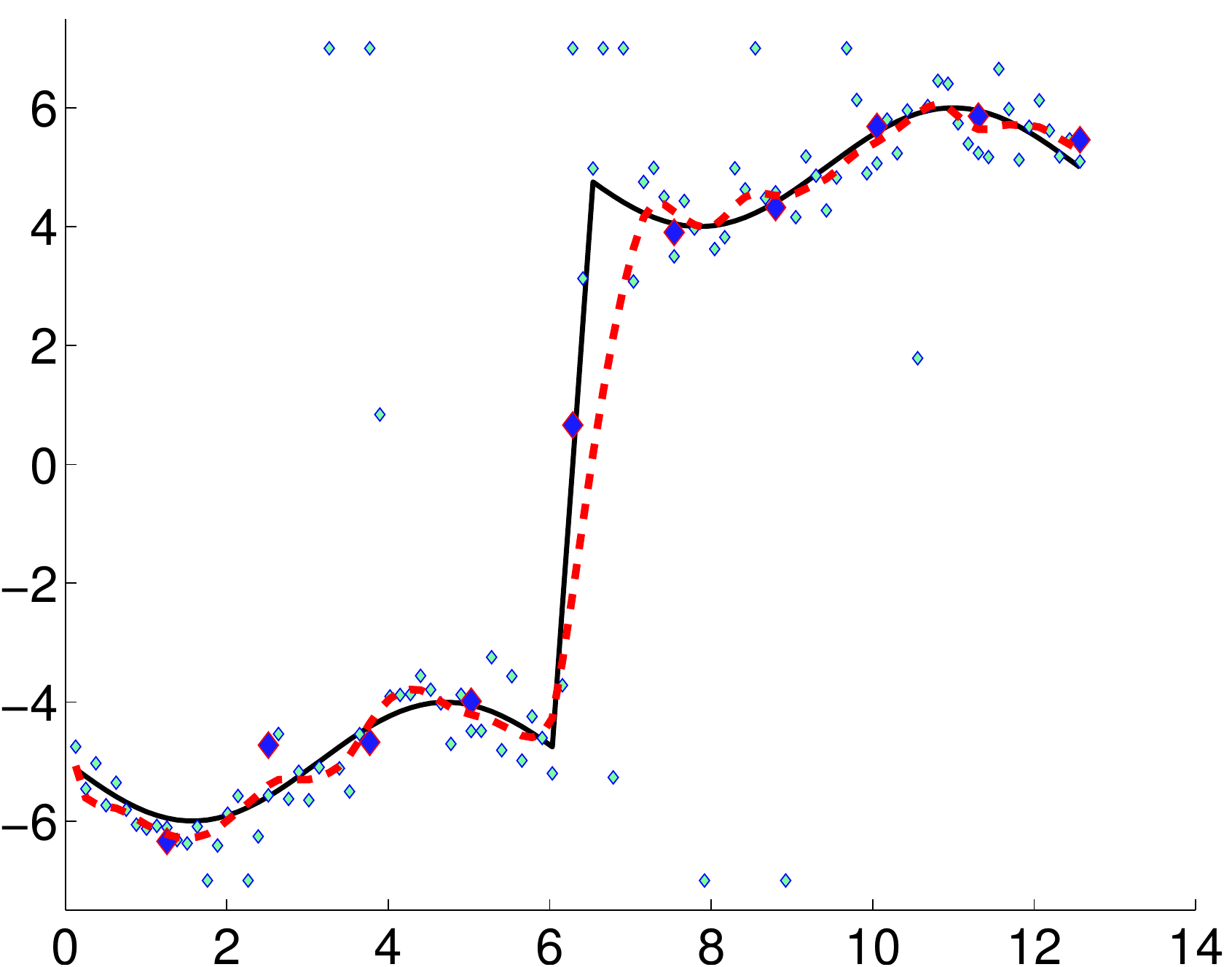}}
(c)
{\includegraphics[scale=0.33]{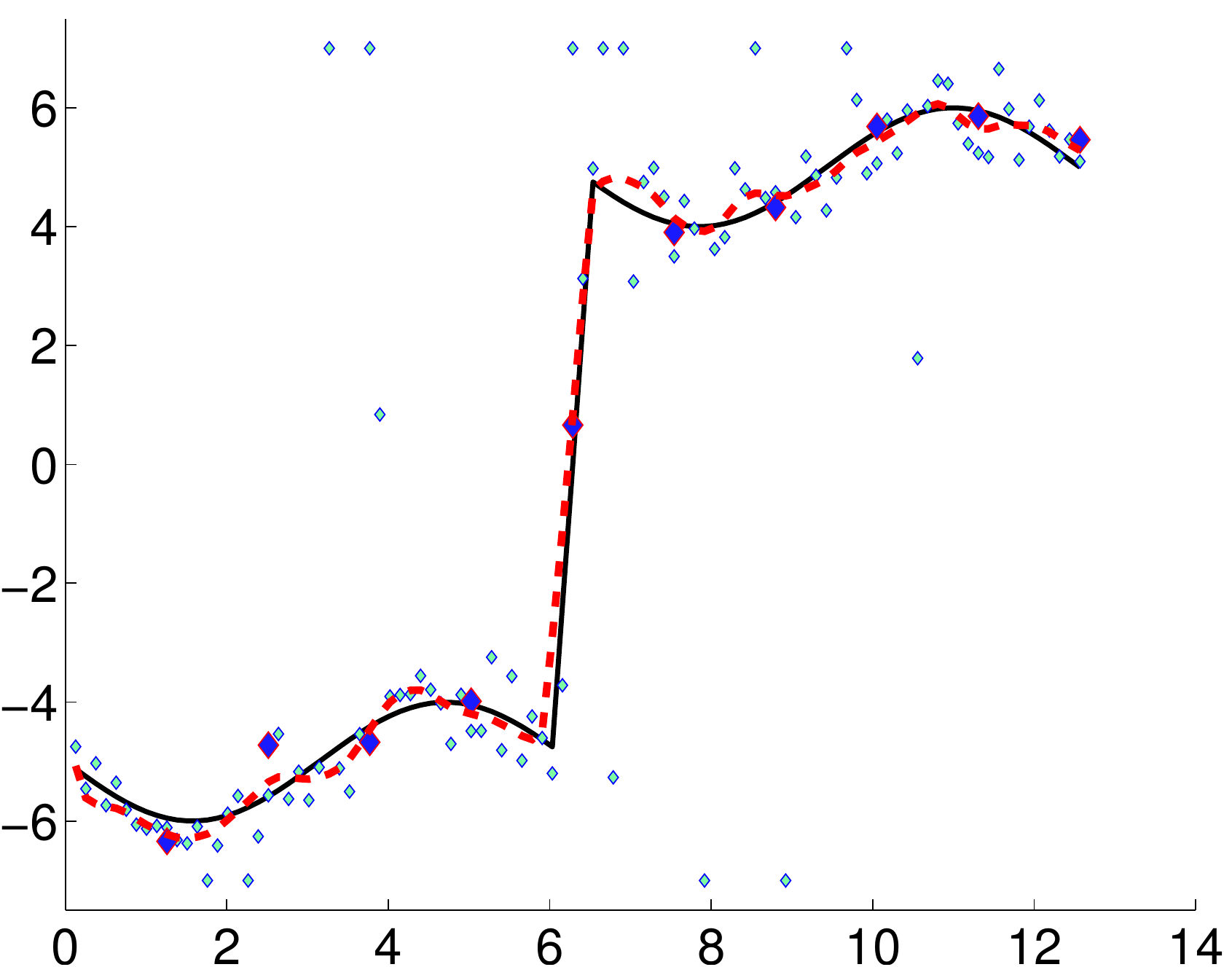}}
(d)

\end{center}
\caption{
    \label{DoubleT}
   Tracking a sudden change in the presence of outliers. Measurements are plotted using diamonds, 
   with $s_2$ measurements (frequent, contaminated) represented by small symbols, while 
	$s_1$ measurements (rare, reliable) represented by large symbols.
	 Outliers appear on the axes when they are out of range of the plot limits. 
   Ground truth is shown using a sold black line, and smoother results are shown
   using a red dashed line. 
   (a): Least squares smoother (Gaussian errors for process and measurements)
   is very vulnerable to outliers.
    (b): T-Robust only (Gaussian errors for process components and $s_1$, Student's t errors for $s_2$) effectively ignores the outliers, 
    but struggles to follow sharp change in state.  
        (c): Double T smoother (Student's t errors for all components): ignores outliers, but struggles to follow sudden change in state.      
        (d): Trend following robust smoother (Student's t errors for process components and $s_2$; Gaussian for $s_1$): ignores outliers and follows sudden change in state, using information in the reliable measurements. 
}
\end{figure*}

%%%%%%%%%%%%%%%%%%%%%%%%%%%%%%%%%%%%%%%%%%%%%%%%%%%%%%%
\section{Discussion and Conclusions}
%%%%%%%%%%%%%%%%%%%%%%%%%%%%%%%%%%%%%%%%%%%%%%%%%%%%%%%

We have presented a generalized Student's t smoothing framework, 
which allows modeling any innovations or measurement residuals 
using Student's t errors, and includes 
T-Robust and T-Trend, and Double T smoothers as important special cases. 
All of the smoothers in the framework efficiently solve for the MAP estimates of
the states in a state-space model
with any selected set of residuals modeled using Student's t or Gaussian noise. 
We have shown that these features can be used independently and in tandem, 
work for linear and nonlinear process models, and can be used both 
for outlier-robust smoothing and for tracking sudden changes in the state. 

Similar to contributions in other applications, e.g. sparse system identification
\cite{Wipf_IEEE_TSP_2007,CPNIPS2010,Tipping2001}, our
results underscore the significant advantages of using
heavy tailed distributions in statistical modeling.
Heavy tailed distributions force the use of non-convex
loss functions to solve for the associated MAP estimates~\cite[Theorem 2]{AravkinFHV:2012}.
The consequent challenge is to
optimize a non-convex objective even when the system dynamics are
linear. In contrast to the convex case, this requires an iterative smoother.
The convergence analysis for these methods is still developed
within the general framework of convex-composite optimization~\cite{Burke85}, although
the details of the analysis differ.

Because the problems are non-convex, iterative methods
may converge to local rather than global minima.
This problem can be mitigated by an appropriate
initialization procedure---for example, in the presence of outliers, the
$\ell_1$-Laplace smoother can be used to obtain a starting point for the optimizer,
in which case we can improve on the $\ell_1$ solution when the data is highly
contaminated with outliers. This approach was not taken in our numerical
experiments, which used the same initial points. 
For all the linear experiments, the initial point was simply the null state sequence. 
For the Van Der Pol, the initial state $x_0$ was correctly specified in all experiments, 
and the remaining state estimates in the initial sequence were null. 

The T-Robust smoother compares favourably to the $\ell_1$-Laplace
smoother described in \cite{AravkinL12011}, and outperforms it
in our experiments when the data is heavily contaminated by outliers.
The T-Trend smoother was designed for tracking signals that may
exhibit sudden changes, and therefore has many potential applications
in areas such as navigation and financial trend tracking. It was demonstrated
to follow a fast jump in the state better than a smoother with a
convex penalty on model deviation. Finally, we demonstrated the 
power of a new method by tracking a fast change in the presence 
of outliers using the full flexibility of the presented framework, 
which allowed us to differentially model residuals for sensors
which we knew to be reliable vs. unreliable, and to design
a smoother that was robust to outliers yet able to track sudden changes.  

An important question in the design and implementation of Student's t-based smoothers
is how to estimate the degree of freedom parameter $\nu$. In all of our experiments,
we have treated this parameter as fixed and know. We note that there
are established EM-based methods in the literature for estimating these parameters~\cite{Lange1989, Fahr1998},
as well as recently proposed methods~\cite{AravkinVanLeeuwen2012}, 
and we leave the implementation of these extensions in the Kalman smoothing
framework to future work.

\section{Acknowledgements}
The authors would like to thank Bradley Bell and North Pacific Acoustic
Laboratory (NPAL) investigators of the Applied Physics Laboratory,
University of Washington for the underwater tracking
data used in this paper (NPAL is sponsored by the Office of
Naval Research code 321OA). We are also grateful to Michael
Gelbart for insightful discussions about the numerical experiments.

%%%%%%%%%%%%%%%%%%%%%%%%%%%%%%%%%%%%%%%%%%%%%%%%%%%%%%%
\bibliographystyle{plain}
\bibliography{biblio}

\begin{thebibliography}{10}

\bibitem{ckbs}
A.Y. Aravkin, B.M. Bell, J.V. Burke, and G.~Pillonetto.
\newblock {CKBS}: {Matlab}/{O}ctave package for constrained and robust {Kalman}
  smoothing.
\newblock {\em http://www.coin-or.org/CoinBazaar/ckbs/ckbs.xml}, 2007-2013.

\bibitem{AravkinL12011}
A.Y. Aravkin, B.M. Bell, J.V. Burke, and G.~Pillonetto.
\newblock An $\ell_1$-{L}aplace robust {K}alman smoother.
\newblock {\em IEEE Transactions on Automatic Control}, 2011.

\bibitem{AravkinIFAC}
A.Y. Aravkin, B.M. Bell, J.V. Burke, and G.~Pillonetto.
\newblock Learning using state space kernel machines.
\newblock In {\em Proc. IFAC World Congress 2011}, Milan, Italy, 2011.

\bibitem{AravkinBellBurkePillonetto2013}
A.Y. Aravkin, B.M. Bell, J.V. Burke, and G.~Pillonetto.
\newblock New stability results and algorithms for block tridiagonal systems,
  with applications to kalman smoothing.
\newblock {\em http://arxiv.org/abs/1303.5237}, 2013.

\bibitem{SYSID2012tks}
A.Y. Aravkin, J.V. Burke, and G.~Pillonetto.
\newblock Robust and trend following kalman smoothers using student's t.
\newblock In {\em Proc. of SYSID}, 2012.

\bibitem{AravkinFHV:2012}
A.Y. Aravkin, M.P. Friedlander, F.~Herrmann, and T.~van Leeuwen.
\newblock Robust inversion, dimensionality reduction, and randomized sampling.
\newblock {\em Mathematical Programming}, 134(1):101--125, 2012.

\bibitem{AravkinVanLeeuwen2012}
A.Y. Aravkin and T.~van Leeuwen.
\newblock Estimating nuisance parameters in inverse problems.
\newblock {\em Inverse Problems}, 28(11):115016, 2012.

\bibitem{Bell2008}
B.~M. Bell, J.~V. Burke, and G.~Pillonetto.
\newblock An inequality constrained nonlinear {Kalman-Bucy} smoother by
  interior point likelihood maximization.
\newblock {\em Automatica}, 2008.

\bibitem{Burke85}
J.V. Burke.
\newblock Descent methods for composite nondifferentiable optimization
  problems.
\newblock {\em Mathematical Programming}, 33:260--279, 1985.

\bibitem{CPNIPS2010}
A.~Chiuso and G.~Pillonetto.
\newblock Learning sparse dynamic linear systems using stable spline kernels
  and exponential hyperpriors.
\newblock In {\em In Advances in Neural Information Processing Systems (NIPS},
  2010.

\bibitem{Chui2009}
Charles Chui and Guanrong Chen.
\newblock {\em Kalman Filtering}.
\newblock Springer, 2009.

\bibitem{Fahr1991}
L.~Fahrmeir and H.~Kaufmann.
\newblock On {K}alman filtering, posterior mode estimation, and {F}isher
  scoring in dynamic exponential family regression.
\newblock {\em Metrika}, pages 37--60, 1991.

\bibitem{Fahr1998}
Ludwig Fahrmeir and Rita Kunstler.
\newblock Penalized likelihood smoothing in robust state space models.
\newblock {\em Metrika}, 49:173--191, 1998.

\bibitem{Farahmand2011}
S.~Farahmand, G.B. Giannakis, and D.~Angelosante.
\newblock Doubly robust smoothing of dynamical processes via outlier sparsity
  constraints.
\newblock {\em Signal Processing, IEEE Transactions on}, 59(10):4529 --4543,
  oct. 2011.

\bibitem{Gelb}
A.~Gelb.
\newblock {\em Applied Optimal Estimation}.
\newblock The M.I.T. Press, Cambridge, MA, 1974.

\bibitem{Hampel}
Frank~R. Hampel, Elvezio~M. Ronchetti, Peter~J. Rousseeuw, and Werner~A.
  Stahel.
\newblock {\em Robust Statistics: The Approach Based on Influence Functions}.
\newblock Wiley Series in Probability and Statistics, 1986.

\bibitem{Hastie01}
T.~J. Hastie, R.~J. Tibshirani, and J.~Friedman.
\newblock {\em The Elements of Statistical Learning. Data Mining, Inference and
  Prediction}.
\newblock Springer, Canada, 2001.

\bibitem{kalman}
R.~E. Kalman.
\newblock A new approach to linear filtering and prediction problems.
\newblock {\em Transactions of the AMSE - Journal of Basic Engineering},
  82(D):35--45, 1960.

\bibitem{Lange1989}
Kenneth~L. Lange, Roderick J.~A. Little, and Jeremy M.~G. Taylor.
\newblock Robust statistical modeling using the t distribution.
\newblock {\em Journal of the American Statistical Association},
  84(408):881--896, 1989.

\bibitem{Mar}
Ricardo~A. Maronna, Douglas Martin, and Yohai.
\newblock {\em Robust Statistics}.
\newblock Wiley Series in Probability and Statistics. Wiley, 2006.

\bibitem{Ohlsson2011}
H.~Ohlsson, F.~Gustafsson, L.~Ljung, and S.~Boyd.
\newblock State smoothing by sum-of-norms regularization.
\newblock {\em Automatica (to appear)}, 2011.

\bibitem{Rock70}
R.~T. Rochafellar.
\newblock {\em Convex Analysis}.
\newblock Princeton University Press, 1970.

\bibitem{RTRW}
R.T. Rockafellar and R.J.B. Wets.
\newblock {\em Variational Analysis}, volume 317.
\newblock Springer, 1998.

\bibitem{Schick1994}
I.C. Schick and S.K. Mitter.
\newblock Robust recursive estimation in the presence of heavy-tailed
  observation noise.
\newblock {\em The Annals of Statistics}, 22(2):1045--1080, June 1994.

\bibitem{Spall03}
J.C. Spall.
\newblock Estimation via {M}arkov chain {M}onte {C}arlo.
\newblock {\em Control Systems Magazine, IEEE}, 23(2):34 -- 45, April 2003.

\bibitem{Lasso1996}
R.~Tibshirani.
\newblock Regression shrinkage and selection via the {LASSO}.
\newblock {\em Journal of the Royal Statistical Society, Series B.}, 58, 1996.

\bibitem{Tipping2001}
M.~Tipping.
\newblock Sparse bayesian learning and the relevance vector machine.
\newblock {\em Journal of Machine Learning Research}, 1:211--244, 2001.

\bibitem{Wahba1990}
G.~Wahba.
\newblock {\em Spline models for observational data}.
\newblock SIAM, Philadelphia, 1990.

\bibitem{West1991}
Mike West and Jeff Harrison.
\newblock {\em Bayesian Forecasting and Dynamic Models}.
\newblock Springer, second edition, 1999.

\bibitem{Wipf_IEEE_TSP_2007}
{D.P.} Wipf and {B.D.} Rao.
\newblock An empirical bayesian strategy for solving the simultaneous sparse
  approximation problem.
\newblock {\em IEEE Transactions on Signal Processing}, 55(7):3704--3716, 2007.

\bibitem{Wright1990}
S.J. Wright.
\newblock Solution of discrete-time optimal control problems on parallel
  computers.
\newblock {\em Parallel Computing}, 16:221--238, 1990.

\end{thebibliography}

\end{document}